\documentclass{amsart}

\usepackage{amssymb}

\usepackage{graphicx}

\usepackage{cite}
\usepackage{psfrag}

\def\FF{\mathbb F}                         
\def\CC{\mathbb C}                         
\def\KK{\mathbb K}                         
\def\FFq2{{\mathbb F}_{q^2}}
\def\rank{\mathop{\mathrm {rk}}\nolimits}  
\def\diag{\mathop{\mathrm {diag}}\nolimits}  
\def\tr{\mathop{\top}\nolimits}     
\def\Tr{\mathop{\mathrm{Tr}}\nolimits}     
\def\N{\mathop{\mathrm{N}}\nolimits}     
\def\hgl{{\mathop{\mathcal{HGL}}\nolimits}_n(\mathbb{F}_{q^2})}
\def\hgldva{{\mathop{\mathcal{HGL}}\nolimits}_2(\mathbb{F}_{q^2})}

\copyrightinfo{}{}

\newtheorem{thm}{Theorem}[section]
\newtheorem{lemma}[thm]{Lemma}
\newtheorem{cor}[thm]{Corollary}
\newtheorem{prop}[thm]{Proposition}

\theoremstyle{definition}

\theoremstyle{remark}

\numberwithin{equation}{section}

\newcommand{\cal}{\mathcal}

\newcommand{\inv}[1]{{\overline{#1}}}

\newcounter{myenumi}

\newsavebox{\cmm}
\savebox{\cmm}{\indent}
\newenvironment{myenumerate}[1]{
\begin{list}{
{\bf #1~\themyenumi}. } {\labelwidth=0pt
\labelsep=0pt\leftmargin=0pt\usecounter{myenumi}} }{\end{list}}

\begin{document}

\title{Adjacency preservers on invertible hermitian matrices~I}

\author[M.~Orel]{Marko Orel}
\address{University of Primorska, FAMNIT, Glagolja\v{s}ka 8, 6000 Koper, Slovenia}
\address{IMFM, Jadranska 19, 1000 Ljubljana, Slovenia}
\address{University of Primorska, IAM, Muzejski trg 2, 6000 Koper, Slovenia}

\email{marko.orel@upr.si}


\begin{abstract}
Hua's fundamental theorem of geometry of hermitian matrices characterizes all bijective maps on the space of all hermitian matrices, which preserve adjacency in both directions. In this and subsequent paper we characterize maps on the set of all invertible hermitian matrices over a finite field, which preserve adjacency in one direction. This and author's previous result are used to obtain two new results related to maps that preserve the `speed of light' on finite Minkowski space-time.

In this first paper it is shown that maps that preserve adjacency on the set of all invertible hermitian matrices over a finite field are necessarily bijective, so the corresponding graph on invertible hermitian matrices, where edges are defined by the adjacency relation, is a core. Besides matrix theory, the proof relies on results from several other mathematical areas, including spectral and chromatic graph theory, and projective geometry.
\end{abstract}

\maketitle

\section{Introduction}

Let $\mathcal{M}$ be a set of matrices over a field (or a division ring) of the same size. Matrices $A,B\in \mathcal{M}$ are \emph{adjacent} if the rank $\rank(A-B)$ is minimal nonzero. A map $\Phi : \mathcal{M}\to \mathcal{M}$  preserves \emph{adjacency in both directions} if $\Phi(A)$ and $\Phi(B)$ are adjacent if and only if $A$ and $B$ are adjacent. We say that $\Phi$ \emph{preserves adjacency} (in one direction) if $\Phi(A)$ and $\Phi(B)$ are adjacent, whenever $A$ and $B$ are adjacent.

In the middle of 20th century Hua~\cite{hua1,hua2,hua3,hua4,hua5,hua6,hua7,hua8} initiated the study of bijective maps $\Phi : \mathcal{M}\to \mathcal{M}$ that preserves adjacency in both directions, for various matrix spaces~$\mathcal{M}$. Characterization of such maps is commonly known as the fundamental theorem of geometry of $\mathcal{M}$~\cite{wan}. These results can be interpreted in the language of graph theory. Let $\Gamma$ be the graph with vertex set $\mathcal{M}$, where  $\{A,B\}$ is an edge if and only if $A$ and $B$ are adjacent. Then the fundamental theorem describes all automorphisms of $\Gamma$, and the maps that preserves adjacency are endomorphisms of $\Gamma$. In our case it will be beneficial to use both styles of terminology simultaneously.

If $\mathcal{M}$ is the set of all hermitian matrices over a field with involution (see Section~2 for detailed definitions), then $A,B\in \mathcal{M}$ are adjacent if $\rank(A-B)=1$. If the involution is not the identity map, then the fundamental theorem says that the automorphisms of the corresponding graph are precisely the maps of the form $\Phi(A)=\lambda PA^{\sigma}P^{\ast}+B$, where $P$ is an invertible matrix with its conjugate transpose~$P^{\ast}$, $B$ is hermitian, $\lambda$ is a nonzero scalar that is fixed by the involution, and $\sigma$ is a field automorphism that commutes with the involution and is applied entry-wise to $A$. Fundamental theorems of geometry of rectangular, alternate, symmetric, and hermitian matrices are summarized in the book~\cite{wan}. In the last decade there were several attempts to reduce the assumptions in these theorems. For hermitian and symmetric matrices over several division rings, preserving adjacency in both directions was reduced to one direction only~\cite{huang_hofer_wan}. In $2\times 2$ case adjacency preservers were characterized even without bijectivity assumption~\cite{huang_W.L.}. For matrices of arbitrary order, such result was obtained for complex hermitian matrices~\cite{huang_semrl} and subsequently for real symmetric matrices~\cite{legisa}.
Optimal version of fundamental theorem, for rectangular matrices over EAS division rings, was proved very recently in a paper with more than 80 pages~\cite{dolgin}. Since the underlying field in~\cite{huang_semrl, legisa, huang_W.L., dolgin} was allowed to be infinite, the removal of bijectivity assumption resulted in an existence of degenerate adjacency preservers, which contain only pairwise adjacent matrices in their images. For finite fields it is not clear in general, if such maps exist. Graph admitting this kind of endomorphisms is said to have a complete core, while graph whose endomorphisms are all automorphisms is said to be a core~\cite{godsil, godsil_knjiga}. In~\cite{FFA_Orel} it was shown that the graph formed by hermitian matrices over a finite field is a core, so its endomorphisms are characterized by the fundamental theorem. Same type of a result was obtained for $n\times n$ symmetric matrices over a finite field if $n\geq 3$, while in $2\times 2$ case the graph  possesses a complete core~\cite{JACO}. Recently it was proved that any (finite) distance transitive graph is either a core or it has a complete core~\cite{godsil}, though it is mentioned that is often difficult to decide which of the two possibilities occurs. It should be noted that graphs formed by all rectangular, alternate, or hermitian matrices over a finite field are all distance transitive~\cite[Section~9.5]{brouwer_cohen_neumaier}.

If $\mathcal{M}$ is just a set and not a vector space, not much of research has been done in this area.
In this and subsequent paper~\cite{obrnljive2} the aim is to characterize all maps~$\Phi$ that preserve adjacency on the set $\hgl$ of all $n\times n$ invertible hermitian matrices over a finite field with $q^2$ elements. We emphasize that we do not assume that $\Phi$ is bijective, while adjacency is assumed to be preserved in just one direction. In fact, in this paper we prove that bijectivity is redundant, that is, the graph formed by the set $\hgl$, which we denote again by $\hgl$, is a core. In the subsequent paper we obtain an analogous result of fundamental theorem, which, in the case of invertible hermitian matrices, is not yet known. Unlike the case of all hermitian matrices over a finite field~\cite{FFA_Orel}, here adjacency preservers do not result to be necessarily affine maps. An example of such is the inverting map $A\mapsto A^{-1}$. Moreover, graph $\hgl$ is not distance transitive if $q\geq 3$, since its maximal cliques contain either $q$ or $q-1$ vertices~(cf. Section~2).

While adjacency preservers have a long history of applications in `preserver problems' (see e.g. survey papers~\cite{survey2,survey3,survey1}), at least two
motives to study such maps on the set $\hgl$ arise from other mathematical areas.

The graph $\hgl$ generalize arguably the most famous graph, i.e., the Petersen graph, which is obtained as the smallest case $n=2=q$. Unfortunately the techniques used in this paper do not work for the two smallest fields. Case~$q=2$ is studied separately, together with the analogous graph formed by invertible binary symmetric matrices, which generalizes the Coxeter graph~\cite{petersen_coxeter}.

A 4-dimensional Minkowski space-time $M_4$ can be described as a 4-dimensional real vector space equipped with an indefinite inner product $$({\bf r}_1,{\bf r}_2):=-x_1x_2-y_1y_2-z_1z_2+c^2t_1t_2,$$
between \emph{events} ${\bf r}_1:=(x_1,y_1,z_1,c t_1)^{\tr}$ and  ${\bf r}_2:=(x_2,y_2,z_2,c t_2)^{\tr}$. Here $c$ denotes the speed of light and units of measurement are chosen such that $c=1$. A light signal can pass between two events ${\bf r}_1$ and  ${\bf r}_2$ if and only if
\begin{equation}\label{light}
({\bf r}_1-{\bf r}_2,{\bf r}_1-{\bf r}_2)=0.
\end{equation}
A map on a Minkowski space-time preserves the speed of light if it preserves relation~(\ref{light}).
Bijective maps that preserve the speed of light in both directions were characterized in~\cite{aleksandrov1950}. It was observed in~\cite{hua_knjiga}  that characterization of such maps is equivalent to the fundamental theorem of geometry of $2\times 2$ complex hermitian matrices. Maps that preserve relation~(\ref{light}) in both directions on an open connected  subset in $M_4$ were characterized in~\cite{aleksandrov,lester}. If this result
is reinterpreted in terms of $2\times 2$ matrices as in~\cite{hua_knjiga}, and applied
for the complement of the light cone $\{{\bf r} : ({\bf r},{\bf r})=0\}$, then we deduce that any map $\Phi:\mathcal{HGL}_2(\CC)\to \mathcal{HGL}_2(\CC)$  that preserves adjacency in both directions, is either of the form $\Phi(A)=\lambda PA^{\sigma}P^{\ast}$ or $\Phi(A)=\lambda P(A^{\sigma})^{-1}P^{\ast}$  for some nonzero real number $\lambda$, a complex invertible matrix~$P$, and a map $\sigma$, which is either the identity map or the complex conjugation. If real numbers are replaced by a finite field, then we obtain a finite analog of Minkowski space-time. For prime fields, bijective maps that preserve relation~(\ref{light}) on finite Minkowski space-time were characterized in~\cite{blasi}. This result is generalized in author's subsequent paper~\cite{obrnljive2} by applying
the characterization of adjacency preservers on hermitian matrices over a finite field. An analogous result for the complement of the light cone in a finite Minkowski space-time is obtained by using the characterization of adjacency preservers on $\hgl$. All these theorems and their connection with special theory of relativity are described in more details in subsequent paper~\cite{obrnljive2}.

We now state the main theorem of this paper in terminology of graph theory.
\begin{thm}\label{glavni}
Let $q\geq 4$ and $n\geq 2$. Then the graph $\hgl$ is a core.
\end{thm}

It turns out that the proof of Theorem~\ref{glavni} is much more involved for $n\geq 4$ than it is in the case $n\in\{2,3\}$. The reason for this is the existence of multidimensional totally isotropic subspaces in $\FFq2^n$ if $n\geq 4$ (see Theorem~\ref{bose_podprostor}). To handle this situation we rely on geometric results on nonexistence of ovoids/spreads in certain hermitian polar spaces and combine them with results from~\cite{cameron_kan, bose}. Another significant step in the proof of Theorem~\ref{glavni} is to show that the core of $\hgl$ is not complete. Here we use several results from spectral and chromatic graph theory. All these lemmas are described in Section~2, which contains also the notation and all definitions. Proof of Theorem~\ref{glavni} is given in Section~3.

\section{Notation and auxiliary theorems}
Here we describe results from various mathematical areas that are used in the proofs. We split them in four subsections.

\subsection*{Finite fields}

Let $\inv{\phantom{a}}$ be an \emph{involution} on a finite field, that is, $\inv{x+y}=\inv{x}+\inv{y}$, $\inv{xy}=\inv{y}\cdot \inv{x}$, and $\inv{\inv{x}}=x$. In this paper we assume that the involution is not the identity map, which implies that the finite field has $q^2$ elements, where $q$ is a power of a prime. We denote the finite field by $\FFq2$ and recall that the unique involution is given by the rule $\inv{x}=x^q$ (cf.~\cite{FFA_Orel}). We use Greek letters for elements of the \emph{fixed field} of the involution $\FF:=\{\lambda\in \FFq2\, :\, \inv{\lambda}=\lambda\}$, which has $q$ elements. The \emph{trace} map $\Tr : \FFq2\to \FF$, given by $\Tr(x)=x+\inv{x}$, is linear if $\FFq2$ and $\FF$ are viewed as vector spaces over $\FF$~\cite[Theorem~2.23]{finite-fields-LN}. Hence the cardinality of the preimage of any $\lambda\in\FF$ is given by
\begin{equation}\label{i1}
|\Tr^{-1}(\lambda)|=q.
\end{equation}
Moreover, \cite[Theorem~2.25]{finite-fields-LN} implies that
\begin{equation}\label{i2}
\Tr(x)=0\qquad \Longleftrightarrow \qquad x=\inv{y}-y\ \emph{for some}\ y\in\FFq2.
\end{equation}
The \emph{norm} map $\N : \FFq2\to \FF$, given by $\N(x)=x\inv{x}$ is surjective. Moreover, it satisfies
\begin{equation}\label{i3}
|\N^{-1}(\lambda)|=q+1
\end{equation}
for any nonzero $\lambda\in\FF$~\cite[Section~2]{bose}. In the proofs, the case $q=4$ needs additional arguments. We mention that the equation
\begin{equation}\label{i4}
\Tr(x)^2+\Tr(x)+\N(x)=0
\end{equation}
has 4 nonzero solutions $x_1, x_2, x_3, x_4$ in the field $\FF_{4^2}$. Moreover, the set
\begin{equation}\label{i44}
\Big\{x_j+\big(\Tr(x_j)+1\big)x_k\ :\ j,k\in\{1,2,3,4\}\Big\}
\end{equation}
consists of 11 nonzero scalars.
The proof of these two facts are left to the reader.

\subsection*{Hermitian matrices}

In this paper $n\geq 2$ is an integer. A $n\times n$ matrix $A$ with coefficients in $\FFq2$ is \emph{hermitian}, if $A^{\ast}:=\inv{A}^{\tr}=A$, where the involution $\inv{\phantom{a}}$ is applied entry-wise and $B^\top$ is the transpose of $B$. If $A$ is hermitian with $\rank A=r$, then
\begin{equation}\label{i8}
A=P\diag(\underbrace{1,\ldots,1}_r,0,\ldots,0)P^{\ast}
\end{equation}
for some invertible matrix $P$ and diagonal matrix $\diag(1,\ldots,1,0,\ldots,0)$~\cite[Theorem~4.1]{bose}. Consequently, $A=\sum_{i=1}^{r} {\bf x}_i{\bf x}_i^{\ast}$, where the column vector ${\bf x}_i\in\FFq2^n$ is the $i$-th column of matrix $P$. Two hermitian matrices $A$ and $B$ are \emph{adjacent} if $\rank(A-B)=1$. In that case, the unique maximal set of pairwise adjacent matrices, containing both $A$ and $B$, is given by
\begin{equation}\label{i5}
\{A+\lambda{\bf x}{\bf x}^{\ast}\ :\ \lambda\in\FF\},
\end{equation}
where ${\bf x}{\bf x}^{\ast}=B-A$~\cite[Corollary~6.9]{wan}. In this paper we use $\hgl$ to denote the set of all $n\times n$ invertible hermitian matrices. A map $\Phi : \hgl\to\hgl$ \emph{preserves adjacency} if $\rank\big(\Phi(A)-\Phi(B)\big)=1$ whenever $\rank(A-B)=1$.

The multiplicativity of the determinant and~(\ref{i8}) imply that $\det A\in\FF$ for any hermitian matrix $A$.
The next lemma evaluate the determinant of an invertible matrix perturbed by a rank-one matrix.
\begin{lemma}{\cite[Chapter~14]{handbook}}\label{handbook2}
Let $\KK$ be a field, $A$ an invertible matrix with coefficients in $\KK$, and ${\bf x},{\bf y}\in \KK^{n}$. Then $$\det(A+{\bf x}{\bf y}^{\tr})=(\det A)\cdot (1+{\bf y}^{\tr}A^{-1}{\bf x}).$$
\end{lemma}

\begin{cor}\label{handbook}
Let $A\in \hgl$, ${\bf x}\in\FF_{q^2}^{n}$, and $\lambda\in\FF$.
Then
\begin{enumerate}
\item $\det(A+\lambda{\bf x}{\bf x}^{\ast})=(\det A)\cdot (1+\lambda{\bf x}^{\ast}A^{-1}{\bf x}),$
\item $A+\lambda{\bf x}{\bf x}^{\ast}$ is invertible if and only if $\lambda{\bf x}^{\ast}A^{-1}{\bf x}\neq -1$,
\item if $A+\lambda{\bf x}{\bf x}^{\ast}$ is invertible, then $$(A+\lambda{\bf x}{\bf x}^{\ast})^{-1}=A^{-1}-\frac{\lambda}{1+\lambda{\bf x}^{\ast}A^{-1}{\bf x}}(A^{-1}{\bf x})(A^{-1}{\bf x})^{\ast}.$$
\end{enumerate}
\end{cor}
\begin{proof}
If we choose ${\bf y}:=\lambda\inv{{\bf x}}$, then  (i) follows from Lemma~\ref{handbook2}, while (ii) follows from (i). A straightforward calculation shows that $$(A+\lambda{\bf x}{\bf x}^{\ast})\left(A^{-1}-\frac{\lambda}{1+\lambda{\bf x}^{\ast}A^{-1}{\bf x}}(A^{-1}{\bf x})(A^{-1}{\bf x})^{\ast}\right)=I,$$
where $I$ is the identity matrix. This proves (iii).
\end{proof}

\subsection*{Graph theory}

All graphs in this paper are finite, undirected, and without loops and multiple edges.
We use $V(\Gamma)$ and $E(\Gamma)$ to denote the vertex set and the edge set of graph $\Gamma$ respectively. A graph $\Gamma'$ is a \emph{subgraph} of $\Gamma$ if $V(\Gamma')\subseteq V(\Gamma)$ and $E(\Gamma')\subseteq E(\Gamma)$. It is an \emph{induced subgraph} if $E(\Gamma')=\{\{u,v\}\in E(\Gamma)\ :\ u,v\in V(\Gamma')\}$.
A \emph{graph homomorphism} between two graphs $\Gamma$ and
$\Gamma'$ is a map $\Phi: V(\Gamma)\to V(\Gamma')$ such that the following implication holds for any
$v,u\in V(\Gamma)$:
\begin{equation}\label{i6}
\{v,u\}\in E(\Gamma) \Longrightarrow \{\Phi(v),\Phi(u)\}\in E(\Gamma').
\end{equation}
In particular, $\Phi(v)\neq \Phi(u)$ for any edge $\{v,u\}$. If $\Phi$ is bijective, the
converse of~(\ref{i6}) also holds, since the graphs are finite. In that case, $\Phi$ is a \emph{graph isomorphism} between $\Gamma$ and $\Gamma'$. If $\Gamma=\Gamma'$, a homomorphism is an \emph{endomorphism}  and an isomorphism is an \emph{automorphism}. A graph $\Gamma$ is a \emph{core} if all its endomorphisms are automorphisms. If $\Gamma$ is a graph, then its subgraph $\Gamma'$ is
a \emph{core of $\Gamma$} if it is a core and there exists
some homomorphism $\Phi : \Gamma\to \Gamma'$. Any
graph has a core, which is always an induced subgraph and unique
up to isomorphism~\cite[Lemma~6.2.2]{godsil_knjiga}. If $\Gamma'$ is a core of $\Gamma$, then there exists a \emph{retraction} $\Psi : \Gamma\to \Gamma'$, that is, a homomorphism, which fixes the subgraph~$\Gamma'$. In fact, if $\Phi : \Gamma\to \Gamma'$ is any homomorphism, then the restriction $\Phi|_{\Gamma'}$ is an automorphism, since $\Gamma'$ is a core. Hence, we can choose $\Psi:=(\Phi|_{\Gamma'})^{-1}\circ\Phi$.

A graph $\Gamma$ is \emph{vertex transitive} if for any pair of vertices $v,u\in V(\Gamma)$ there exists an automorphism $\Phi$ of $\Gamma$ such that $\Phi(v)=u$. A core of a vertex transitive graph is vertex transitive~\cite[Theorem~6.13.1]{godsil_knjiga}.

Given a
graph $\Gamma$, we use $\chi(\Gamma)$ to denote its \emph{chromatic
number}, that is, the smallest integer~$m$ for which there exists a
\emph{vertex $m$-coloring}, i.e., a map ${\cal C}:
V(\Gamma)\to\{1,2,\ldots,m\}$  that satisfies ${\cal C}(v)\neq
{\cal C}(u)$ whenever $\{v,u\}$ is an edge.
The next simple lemma is
well known (see e.g.~\cite[Proposition~1.20]{klavzar}
or~\cite[Corollary~1.8]{homomorfizmi}). It just observes that a
composition of a graph homomorphism and a vertex coloring is still a
vertex coloring.
\begin{lemma}\label{lema-kromaticno}
If there is some graph homomorphism $\Phi: \Gamma\to\Gamma'$, then $\chi(\Gamma)\leq\chi(\Gamma')$.
\end{lemma}
The \emph{adjacency matrix} of a graph $\Gamma$ with vertex set $V(\Gamma)=\{u_1, u_2,\ldots, u_t\}$ is an $n\times n$ binary matrix, whose $(i,j)$-th entry equals 1 if $\{u_i,u_j\}$ is an edge and 0 otherwise. Being a real symmetric matrix, all its eigenvalues are real. We denote them in descending order by $\lambda_1\geq \lambda_2\geq\ldots\geq \lambda_t$. Since the adjacency matrix has zero diagonal, it follows that
\begin{equation}\label{i10}
\lambda_1+\lambda_2+\ldots+\lambda_t=0.
\end{equation}
When $\Gamma$ is \emph{regular}, that is, each vertex has the same number $k$ of neighbors, then~\cite[Proposition~3.1]{biggs} implies that
\begin{equation}\label{i11}
\lambda_1=k.
\end{equation}
Another property of graph eigenvalues is given by Theorem~\ref{haemers}~\cite[Proposition~3.6.3]{spekter}.
\begin{thm}\label{haemers}
If $t>\chi(\Gamma)$, then $\lambda_2+\lambda_3+\ldots+ \lambda_{\chi(\Gamma)}+\lambda_{t-\chi(\Gamma)+1}\geq 0$.
\end{thm}
If $\Gamma'$ is an induced subgraph of $\Gamma$ on $t-1$ vertices, say $u_1, u_2,\ldots, u_{t-1}$, then it is well known (cf.~\cite[Proposition~3.2.1]{spekter} or~\cite[Section~6.5]{handbook2}) that its eigenvalues $\mu_1,\ldots,\mu_{t-1}$ \emph{interlace} those of $\Gamma$, that is,
\begin{equation}\label{i7}
\lambda_1\geq \mu_1 \geq\lambda_2\geq \mu_2\geq \ldots\geq \lambda_{t-1}\geq \mu_{t-1} \geq\lambda_t.
\end{equation}

The graph considered in this paper has $\hgl$ as its vertex set  and $\{\{A,B\}\ :\ A,B\in \hgl, \rank(A-B)=1\}$ as its edge set. We slightly abuse the notation and denote the graph by $\hgl$ as well. Any two adjacent vertices in $\hgl$ lies in a unique \emph{maximal clique}, i.e. a maximal set of pairwise adjacent vertices (cf.~(\ref{i5})). If $q\geq 3$, a maximal clique has either $q$ or $q-1$ vertices (cf.~Figure~1). In fact, Corollary~\ref{handbook} implies that all matrices in~(\ref{i5}) are invertible if ${\bf x}^{\ast}A^{-1}{\bf x}=0$, while only one is singular if ${\bf x}^{\ast}A^{-1}{\bf x}\neq 0$.
In this paper, a maximal clique with $q-1$~vertices  is called  \emph{$(q-1)$-clique}.  A \emph{$q$-clique} is defined analogously. By Corollary~\ref{handbook}, all matrices in a $q$-clique have the same determinant, while the determinants of matrices in a $(q-1)$-clique are all distinct. A map $\Phi: \hgl\to\hgl$ preserves adjacency precisely when it is an endomorphism of graph $\hgl$.
\begin{figure}
\centering
\includegraphics[width=0.5\textwidth]{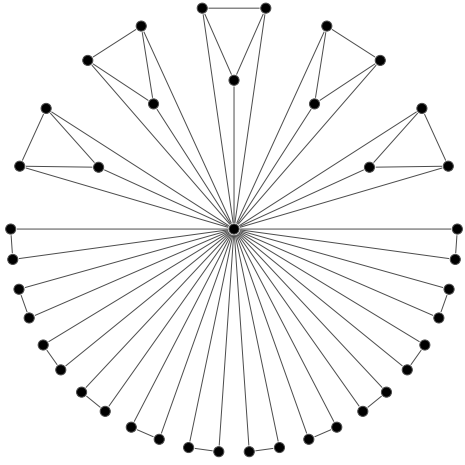}
\caption{ Neighborhood of any vertex in ${\cal HGL}_2(\FF_{4^2})$. Matrices in a $4$-clique have the same determinant, determinants of matrices in a $3$-clique are all distinct.}
\end{figure}
It follows from~(\ref{i8}) that for any pair of matrices $A,B\in \hgl$ there is an invertible matrix $P$ such that $PAP^{\ast}=B$. So graph $\hgl$ and its core are vertex transitive, since the map $X\mapsto PXP^{\ast}$ is an automorphism of $\hgl$.

In the proofs section, a lower bound for the chromatic number of $\hgl$ is provided. The bound is computed with a help of spectral information on the graph ${\cal H}_2(\FFq2)$ that is formed analogously as $\hgl$, except that the vertex set consists of all $2\times 2$ hermitian matrices (singular matrices are included). It follows from~\cite[Theorem~9.5.7 and the discussion on pp. 128--129]{brouwer_cohen_neumaier} (see~\cite{FFA_Orel} as well) that ${\cal H}_2(\FFq2)$ is a distance-regular graph of diameter two and its distinct eigenvalues are precisely the eigenvalues of the matrix
$$\begin{bmatrix}
a_0&b_0&0\\
c_1&a_1&b_1\\
0&c_2&a_2
\end{bmatrix},$$
where
\begin{equation*}
c_i=\frac{q^{i-1}(q^i-(-1)^i)}{q+1},\quad a_i=\frac{q^{2i}-q^{i-1}(q^i-(-1)^i)-1}{q+1},\quad b_i=\frac{q^{4}-q^{2i}}{q+1}.
\end{equation*}
These eigenvalues are $k:=q^3-q^2+q-1$, $r:=q-1$, and $s:=-q^2+q-1$. Their multiplicities, in the same order, equal  $1$, $\frac{(s+1)k(k-s)}{c_2(s-r)}=q^4-k-1$, and $q^4-1-\frac{(s+1)k(k-s)}{c_2(s-r)}=k$~\cite[Theorem~1.3.1~(vi)]{brouwer_cohen_neumaier}.

\subsection*{Projective geometry}

A \emph{hermitian variety} of $n\times n$ hermitian matrix $A$ over $\FFq2$ is defined by
$$
V_A=\{\langle{\bf x}\rangle\ :\ {\bf x}^{\tr}A\inv{{\bf x}}=0, 0\neq {\bf x}\in \FFq2^n \},
$$
where $\langle{\bf x}\rangle$ is the 1-dimensional subspace in $\FFq2^n$ generated by vector~${\bf x}$~\cite{bose}. In this paper we prefer the term ${\bf x}^{\ast}A{\bf x}$ instead of ${\bf x}^{\tr}A\inv{{\bf x}}$. If ${\bf x}\neq 0$, then
\begin{equation}\label{i12}
{\bf x}^{\ast}A{\bf x}=0 \Longleftrightarrow \langle{\bf x}\rangle\in V_{\inv{A}}.
\end{equation}
The next theorem is proved in~\cite[Theorem 8.1 and its corollary]{bose}.
\begin{thm}\label{kardinalnost}
If $\rank A=r$, then the cardinality of the variety equals
$$|V_A|=\frac{q^{2n-1}+(-1)^r(q-1)q^{2n-r-1}-1}{q^2-1}.$$
In particular, $V_A$ is nonempty.
\end{thm}
Given a  vector subspace ${\cal U}\subseteq \FFq2^n$, we write ${\cal U}\subseteq V_A$ if $\langle {\bf x}\rangle\in V_A$ for all nonzero ${\bf x}\in {\cal U}$. In this case it is easy to see that
\begin{equation}\label{opomba}
{\bf x}^{\ast}\inv{A}{\bf y}=0
\end{equation}
for all  ${\bf x}, {\bf y} \in {\cal U}$ (cf.~\cite[Lemma~9.1]{bose}). We use $\dim {\cal U}$ to denote the vector space dimension of ${\cal U}$ (one more than the projective dimension).
\begin{thm}{\cite[Theorem~9.1]{bose}}\label{bose_podprostor}
Let $A\in \hgl$ and suppose that  ${\cal U}\subseteq \FFq2^n$ is a vector subspace.
If ${\cal U}\subseteq V_A$, then $\dim {\cal U}\leq \big\lfloor\frac{n}{2}\big\rfloor$.
\end{thm}
Let $n\geq 4$. Given an invertible hermitian matrix $A$, the \emph{point graph of a hermitian polar space} is a graph, here denoted by $H(n-1,q^2)$, with vertex set $V_A$, where two distinct vertices $\langle{\bf x}\rangle$ and $\langle{\bf y}\rangle$ are adjacent if and only if~(\ref{opomba}) holds. A subspace~${\cal U}$ from Theorem~\ref{bose_podprostor} is \emph{maximal totally isotropic} if $\dim {\cal U}= \big\lfloor\frac{n}{2}\big\rfloor$. An \emph{ovoid of a hermitian polar space} is a subset ${\cal O}\subseteq V_A$ meeting every maximal totally isotropic subspace in one point, that is, any maximal totally isotropic subspace ${\cal U}$ satisfies ${\bf x}\in {\cal U}$ for precisely one  $\langle{\bf x}\rangle\in{\cal O}$. A \emph{spread of a hermitian polar space} is a set $\{{\cal U}_1,\ldots ,{\cal U}_s\}$ of maximal totally isotropic subspaces that satisfies $\bigcup_{i=1}^{s}\{\langle{\bf x}\rangle\ :\ 0\neq {\bf x}\in {\cal U}_i\}=V_A$ and ${\cal U}_i\cap {\cal U}_j=\{0\}$ for $i\neq j$. Recall that given a graph $\Gamma$, its \emph{complement} $\inv{\Gamma}$ is a graph on the same vertex set as $\Gamma$ with $\{v,u\}\in E(\inv{\Gamma})$ if and only if $\{v,u\}\notin E(\Gamma)$. It was proved in~\cite[Theorem~3.5, Corollary~2.2]{cameron_kan} that the complement of $H(n-1,q^2)$ is a core if either an ovoid or a spread does not exist. In~\cite[Theorem~20]{thas} it was proved that $H(n-1,q^2)$ has no spread if $n-1\geq 5$ is odd. By~\cite[3.2.3 and 3.4.1(ii)]{payne_thas},  $H(3,q^2)$ has no spread as well.  In~\cite{thas2} it was proved that $H(n-1,q^2)$ has no ovoid if $n-1\geq 4$ is even.
Hence, the following lemma holds.
\begin{lemma}{(cf.~\cite{cameron_kan})}\label{cameron}
Let $n\geq 4$. The complement of $H(n-1,q^2)$ is a core.
\end{lemma}
We refer to~\cite{thas_hand,cameron_hand} for more on polar spaces and to~\cite{cameron_kan} for a very short introduction to this subject.

\section{Proofs}

In all results we assume $q\geq 2$  and $n\geq 2$ unless otherwise stated.
We start with a technical lemma about the field $\FFq2$.

\begin{lemma}\label{obseg_trace}
Let $a,b,x\in\FFq2$ with $a,b\neq 0$. Then there exists at most one scalar $y\in\FFq2\backslash\{x\}$ such that $\N(y)=\N(x)$ and $\N(a+by)=\N(a+bx)$.
\end{lemma}
\begin{proof}
Any such $y$ satisfies $\Tr\big(\inv{a}b(x-y)\big)=0$. By~(\ref{i2}) there is $z\in \FFq2$ such that
\begin{equation}\label{eq25}
\inv{a}b(x-y)=\inv{z}-z.
\end{equation}
The equation $\N(\inv{a}bx)=\N(\inv{a}by)$ transforms into $(\inv{z}-z)^2=(a\inv{b}\inv{y}-\inv{a}by)(\inv{z}-z)$. Since  $\inv{z}-z\neq 0$, we deduce that $\inv{z}-z=a\inv{b}\inv{y}-\inv{a}by$. Hence,  $y=\frac{a\inv{b}}{\inv{a}b}\inv{x}$ by~(\ref{eq25}).
\end{proof}

We continue with three results on column vectors. We say that  ${\bf x}_1,\ldots,{\bf x}_k\in \FFq2^n$ are \emph{orthonormal} if ${\bf x}_i^{\ast}{\bf x}_i=1$ for all $i$ and ${\bf x}_i^{\ast}{\bf x}_j=0$ if $i\neq j$. Orthonormal vectors are linearly independent.
\begin{prop}\label{ortonormalnost}
Let $1\leq j<n$. If ${\bf x}_1,\ldots,{\bf x}_j\in \FFq2^n$ are orthonormal, there exist vectors ${\bf x}_{j+1},\ldots,{\bf x}_n$ such that ${\bf x}_1,\ldots,{\bf x}_n$ is an orthonormal basis of~$\FFq2^n$.
\end{prop}
\begin{proof}
It suffices to find a vector ${\bf x}_{j+1}$ such that ${\bf x}_{j+1}^{\ast}{\bf x}_{j+1}=1$ and ${\bf x}_{i}^{\ast}{\bf x}_{j+1}=0$ for $i\leq j$. Then the proof ends by an induction.

Consider the $j\times n$ matrix $X$ with ${\bf x}_i^{\ast}$ as the $i$-th row for all $i\leq j$. The kernel $\ker X$ is $n-j$ dimensional. If it contains a nonzero vector ${\bf y}_{j+1}$ with $\langle{\bf y}_{j+1}\rangle\notin V_I$, then we choose a nonzero $a\in \FFq2$ such that $\N(a)=({\bf y}_{j+1}^{\ast}{\bf y}_{j+1})^{-1}$, and define ${\bf x}_{j+1}:=a{\bf y}_{j+1}$. We will show that such ${\bf y}_{j+1}$ indeed exists.

Let ${\bf z}_1,\ldots,{\bf z}_{n-j}$ be a basis of $\ker X$. We claim that  ${\bf z}_1,\ldots,{\bf z}_{n-j},{\bf x}_1,\ldots,{\bf x}_{j}$  are linearly independent vectors.
In fact, if $\sum_{s=1}^{n-j} c_s {\bf z}_s+\sum_{t=1}^{j} d_t {\bf x}_t=0$, then a multiplication by ${\bf x}_{t}^{\ast}$ from the left implies that $0=d_t {\bf x}_{t}^{\ast}{\bf x}_t=d_t$ for all $t$. Hence, $\sum_{s=1}^{n-j} c_s {\bf z}_s=0$. Since ${\bf z}_1,\ldots,{\bf z}_{n-j}$ are linearly independent, we deduce that $c_s=0$ for all $s$. This proves that matrix $Y$, with ${\bf x}_1^{\ast},\ldots,{\bf x}_{j}^{\ast}$ as first $j$ rows and ${\bf z}_1^{\ast},\ldots,{\bf z}_{n-j}^{\ast}$ as last $n-j$ rows, is invertible. If $\ker X\subseteq V_I$, then~(\ref{opomba}) implies that $\ker X\subseteq \ker Y=\{0\}$, a contradiction. Hence,  $\langle{\bf y}_{j+1}\rangle\notin V_I$ for some ${\bf y}_{j+1}\in\ker X$.
\end{proof}

\begin{lemma}\label{lemma-mesano}
Suppose that ${\bf x}_1,\ldots,{\bf x}_4\in \FFq2^n$ satisfy ${\bf x}_i^{\ast}{\bf x}_i=0$ for all $i$ and assume that ${\bf x}_1^{\ast}{\bf x}_4\neq 0$. Then there exist $a_2,a_3,a_4\in\FFq2$ not all zero such that
\begin{equation}\label{eq32}
(a_1{\bf x}_1+a_2{\bf x}_2+a_3{\bf x}_3+a_4{\bf x}_4)^{\ast}(a_1{\bf x}_1+a_2{\bf x}_2+a_3{\bf x}_3+a_4{\bf x}_4)=0
\end{equation}
and
\begin{equation}\label{eq33}
{\bf x}_1^{\ast}(a_1{\bf x}_1+a_2{\bf x}_2+a_3{\bf x}_3+a_4{\bf x}_4)=0
\end{equation}
for all $a_1\in \FFq2$.
\end{lemma}
\begin{proof}
Consider the $2\times 2$ hermitian matrix
$$A=\begin{bmatrix}
-\Tr\big(\frac{{\bf x}_1^{\ast}{\bf x}_2\cdot {\bf x}_2^{\ast}{\bf x}_4}{{\bf x}_1^{\ast}{\bf x}_4}\big)&{\bf x}_2^{\ast}{\bf x}_3-\frac{{\bf x}_4^{\ast}{\bf x}_3\cdot {\bf x}_2^{\ast}{\bf x}_1}{{\bf x}_4^{\ast}{\bf x}_1}-\frac{{\bf x}_1^{\ast}{\bf x}_3\cdot {\bf x}_2^{\ast}{\bf x}_4}{{\bf x}_1^{\ast}{\bf x}_4}\\
{\bf x}_3^{\ast}{\bf x}_2-\frac{{\bf x}_3^{\ast}{\bf x}_4\cdot {\bf x}_1^{\ast}{\bf x}_2}{{\bf x}_1^{\ast}{\bf x}_4}-\frac{{\bf x}_3^{\ast}{\bf x}_1\cdot {\bf x}_4^{\ast}{\bf x}_2}{{\bf x}_4^{\ast}{\bf x}_1}&-\Tr\big(\frac{{\bf x}_1^{\ast}{\bf x}_3\cdot {\bf x}_3^{\ast}{\bf x}_4}{{\bf x}_1^{\ast}{\bf x}_4}\big)
\end{bmatrix}.$$
By Theorem~\ref{kardinalnost}, $V_{\inv{A}}\neq \emptyset$, so there exist $(a_2,a_3)\neq (0,0)$ such that \begin{equation}\label{eq31}
[\begin{smallmatrix} a_2\\
a_3\end{smallmatrix}]^{\ast} A [\begin{smallmatrix} a_2\\
a_3\end{smallmatrix}]=0.
\end{equation}
Choose $a_4=-\frac{a_2{\bf x}_1^{\ast}{\bf x}_2+a_3{\bf x}_1^{\ast}{\bf x}_3}{{\bf x}_1^{\ast}{\bf x}_4}$. Then (\ref{eq33}) is satisfied. Moreover, a straightforward computation with a use of (\ref{eq31}) shows that
$(a_2{\bf x}_2+a_3{\bf x}_3+a_4{\bf x}_4)^{\ast}(a_2{\bf x}_2+a_3{\bf x}_3+a_4{\bf x}_4)=0$, so (\ref{eq33}) implies (\ref{eq32}).
\end{proof}

\begin{lemma}\label{det-tenzor}
If $\alpha_1,\ldots, \alpha_n\in\FF$ and  ${\bf x}_1,\ldots,{\bf x}_n\in\FFq2^n$, then
$$\det\left(\sum_{i=1}^n\alpha_i{\bf x}_i{\bf x}_i^{\ast}\right)=\det\left(\sum_{i=1}^n{\bf x}_i{\bf x}_i^{\ast}\right)\prod_{i=1}^n\alpha_i .$$
\end{lemma}
\begin{proof}
Pick $c_i\in\FFq2$ such that $\N(c_i)=\alpha_i$.
Let $Q$ and $P$ be $n\times n$ matrices with $c_i{\bf x}_i$  and ${\bf x}_i$
as the $i$-th column respectively. Then $\det Q=\det P\cdot \prod_{i=1}^n c_i$. If ${\bf e}_i$ denotes the $i$-th standard vector, then
$\sum_{i=1}^n\alpha_i{\bf x}_i{\bf x}_i^{\ast}=\sum_{i=1}^n Q{\bf e}_i{\bf e}_i^{\ast}Q^{\ast}=QQ^{\ast}$ and $\sum_{i=1}^n{\bf x}_i{\bf x}_i^{\ast}=PP^{\ast}$, so the proof ends by multiplicativity of the determinant.
\end{proof}

The next four results involve graph $\hgl$.

\begin{prop}\label{lemma-povez-det}
Let $q\geq 4$ and let $\lambda\in\FF$ be nonzero. The subgraph $\Gamma$ in $\hgl$, which is induced by the preimage $\det^{-1}(\lambda)$, is connected.
\end{prop}
\begin{proof}
We split the proof in three steps.
\begin{myenumerate}{Step}
\item If ${\bf y}_1^{\ast}{\bf y}_1=1={\bf y}_2^{\ast}{\bf y}_2$,  then there is a walk in $\Gamma$ that connects $I+(\lambda-1){\bf y}_1{\bf y}_1^{\ast}$ with $I+(\lambda-1){\bf y}_2{\bf y}_2^{\ast}$.

    Determinants of $A_1:=I+(\lambda-1){\bf y}_1{\bf y}_1^{\ast}$ and $A_2:=I+(\lambda-1){\bf y}_2{\bf y}_2^{\ast}$ equal $\lambda$  by Corollary~\ref{handbook}. The proof is trivial if ${\bf y}_1, {\bf y}_2$ are linearly dependent, so assume they are independent. We can assume that $\lambda\neq 1$.
    Three cases are considered.
    \medskip

    \emph{Case~1.} Assume that ${\bf y}_1^{\ast}{\bf y}_2=0$. Pick $a_1,a_2\in \FFq2$ such that $\N(a_1)=-\lambda$ and $\N(a_2)=1$. Corollary~\ref{handbook} implies that
        \begin{align}
        \nonumber (a_1{\bf y}_1+a_2{\bf y}_2)^{\ast}&A_1^{-1}(a_1{\bf y}_1+a_2{\bf y}_2)=\\
        \label{eq26}&=(a_1{\bf y}_1+a_2{\bf y}_2)^{\ast}\big(I+(\lambda^{-1}-1){\bf y}_1{\bf y}_1^{\ast}\big)(a_1{\bf y}_1+a_2{\bf y}_2)=0,
        \end{align}
        so matrix $B_1:=A_1+(a_1{\bf y}_1+a_2{\bf y}_2)(a_1{\bf y}_1+a_2{\bf y}_2)^{\ast}$ has determinant $\lambda$ by Corollary~\ref{handbook}. Similarly we see that $B_2:=A_2+(b_1{\bf y}_1+b_2{\bf y}_2)(b_1{\bf y}_1+b_2{\bf y}_2)^{\ast}$ has determinant $\lambda$ for $b_1:=a_1 c$ and $b_2:=a_2\inv{c^{-1}}$, where $\N(c)=\lambda^{-1}$. Since $B_1=B_2+2(1-\lambda){\bf y}_2{\bf y}_2^{\ast},$ we see that $B_1$ and $B_2$ are either adjacent or the same. So $A_1,B_1,B_2,A_2$ or $A_1,B_1, A_2$ is the desired walk.\medskip

     \emph{Case~2.} Assume that ${\bf y}_1^{\ast}{\bf y}_2\cdot {\bf y}_2^{\ast}{\bf y}_1\notin\{0,1\}$. Choose $a_2:=\frac{1}{{\bf y}_1^{\ast}{\bf y}_2}$. It follows from the assumption that $\lambda(1-\N(a_2))\neq 0$, so by (\ref{i3}) there exist $q+1>2$ scalars $a_1$ such that $\N(1+a_1)=\lambda(1-\N(a_2))$. By Lemma~\ref{obseg_trace} we may choose $a_1$ such that $\N(a_1)\neq \N(a_2)$. A straightforward calculation shows that (\ref{eq26}) holds, so $\det\big(A_1+\mu(a_1{\bf y}_1+a_2{\bf y}_2)(a_1{\bf y}_1+a_2{\bf y}_2)^{\ast}\big)=\lambda$ for all $\mu\in\FF$. Choose $\mu:=\frac{\lambda-1}{\N(a_2)-\N(a_1)}$ and denote the corresponding matrix by $B$. A straightforward calculation shows that $B$ is adjacent to $A_2$, so $A_1,B,A_2$ is the desired walk.

    \emph{Case~3.} Assume that ${\bf y}_1^{\ast}{\bf y}_2\cdot {\bf y}_2^{\ast}{\bf y}_1=1$. It suffices to show that there is a walk in $\Gamma$ that connects $A_i$ with $C:=I+(\lambda-1){\bf e}_1{\bf e}_1^{\ast}$. Let $i\in\{1,2\}$ and denote ${\bf y}_i=:(u_1,\ldots,u_n)^{\tr}$. If $\N(u_1)\neq 1$, then ${\bf y}_i^{\ast} {\bf e}_1{\bf e}_1^{\ast} {\bf y}_i\neq 1$, so we find a walk in $\Gamma$ from $A_i$ to $C$ in the same way as we constructed the desired walk in Cases~1 and~2. Let $\N(u_1)=1$. There is nothing to prove if ${\bf y}_i{\bf y}_i^{\ast}={\bf e}_1{\bf e}_1^{\ast}$, so we may assume that $u_k\neq 0$ for some $k\geq 2$. To simplify the writings let $k=2$.
    Since $\N(u_1)=1$, ${\bf y}_i^{\ast}{\bf y}_i=1$, and $u_2\neq 0$, it follows that $n\geq 3$.
     Since $q\geq 3$ there is $v_1\in\FFq2$ such that $\N(v_1)\notin\{1,\N(u_1)+\N(u_2)\}$ and $v_1\neq -(\inv{u}_1)^{-1}\cdot\sum_{i=3}^n \N(u_i)$. By (\ref{i3}) there are $q+1>2$ scalars $v_2\in\FFq2$ such that
    \begin{equation}\label{eq27}
    \N(v_2)=\N(u_1)+\N(u_2)-\N(v_1).
    \end{equation}
    Since $v_1\inv{u}_1+\sum_{i=3}^n \N(u_i)\neq 0\neq \inv{u}_2$, we can apply
    Lemma~\ref{obseg_trace} and choose $v_2$ with
     \begin{equation}\label{eq28}
    \N\left(\big(v_1\inv{u}_1+\sum_{i=3}^n \N(u_i)\big)+v_2\inv{u}_2\right)\neq 1.
    \end{equation}
    Let ${\bf z}:=(v_1,v_2,u_3,\ldots,u_n)^{\tr}$. Then (\ref{eq27})-(\ref{eq28}) imply that ${\bf z}^{\ast}{\bf z}=1$ and ${\bf z}^{\ast}{\bf y}_i{\bf y}_i^{\ast}{\bf z}\neq 1$. Since ${\bf y}_i^{\ast}{\bf y}_i=1$, we can find a walk in $\Gamma$ that connects $A_i$ with $I+(\lambda -1){\bf z}{\bf z}^{\ast}$ by a similar procedure as in Cases 1 and 2.   Since ${\bf z}^{\ast}{\bf e}_1{\bf e}_1^{\ast}{\bf z}=\N(v_1)\neq 1$, we similarly see  that there is a walk in $\Gamma$ from $I+(\lambda -1){\bf z}{\bf z}^{\ast}$ to $C$, so there is a walk in $\Gamma$ from $A_i$ to $C$. This ends the proof of Step~1.

\item If $\det A_1=\lambda =\det A_2$ and $B\in\hgl$ is adjacent to $A_1$ and $A_2$, then there is a walk in $\Gamma$ that connects $A_1$ with $A_2$.

    If $\det B=\lambda$, we are done. So let $\det B\neq \lambda$.
    Pick an invertible $P$ such that $B=PP^{\ast}$. Then $A_i=P(I+{\bf x}_i{\bf x}_i^{\ast})P^{\ast}$ for some column vector ${\bf x}_i$. By Corollary~\ref{handbook} and the multiplicativity of the determinant function, we deduce that $1+{\bf x}_i^{\ast}{\bf x}_i=\det(I+{\bf x}_i{\bf x}_i^{\ast})=\frac{\lambda}{\det{B}}$, so ${\bf x}_i^{\ast}{\bf x}_i=\frac{\lambda}{\det{B}}-1\neq 0$.
    Let ${\bf y}_i:=c{\bf x}_i$, where $c\in\FFq2$ satisfies  $\N(c)=(\frac{\lambda}{\det{B}}-1)^{-1}$. Then ${\bf y}_i^{\ast}{\bf y}_i=1$ and $I+{\bf x}_i{\bf x}_i^{\ast}=I+(\frac{\lambda}{\det{B}}-1){\bf y}_i{\bf y}_i^{\ast}$. By Step~1, there is a walk between $I+(\frac{\lambda}{\det{B}}-1){\bf y}_1{\bf y}_1^{\ast}$ and $I+(\frac{\lambda}{\det{B}}-1){\bf y}_2{\bf y}_2^{\ast}$ in the graph that is induced by $\det^{-1} (\frac{\lambda}{\det{B}})$. If this walk is premultiplied by $P$ and postmultiplied by $P^{\ast}$, we obtain a walk between $A_1$ and $A_2$ that lies in~$\Gamma$.

\item If $\det A_1=\lambda =\det A_2$, then there is a walk in $\Gamma$ between $A_1$ and $A_2$.

     By Step~2, it suffices to construct matrices $B_0,B_1,\ldots,B_{2m}$ such that
     \begin{equation}
     \label{eq30} B_0=A_1,\ B_{2m}=A_2,\ \rank(B_{i+1}-B_i)\leq 1,\ \det B_{2i}=\lambda.
     \end{equation}
     for all $i$. Write $A_1=\sum_{i=1}^n{\bf x}_i{\bf x}_i^{\ast}$ and $A_2=\sum_{i=1}^n{\bf y}_i{\bf y}_i^{\ast}$ for suitable column vectors ${\bf x}_i$ and ${\bf y}_i$. Obviously, both ${\bf x}_1,\ldots,{\bf x}_n$ and ${\bf y}_1,\ldots,{\bf y}_n$ form a basis in~$\FFq2^n$.
     Let $P$ be the invertible matrix with ${\bf x}_i$ as the $i$-th column, so $A_1=PP^{\ast}$.
     Let $B_0:=A_1$. Since $q\geq 4$, Corollary~\ref{handbook} shows that there are two distinct nonzero $\mu_1,\mu_2\in\FF$ such that $B_0+\mu_i {\bf y}_1{\bf y}_1^{\ast}$ is invertible. In the set $\{{\bf x}_1,\ldots,{\bf x}_n\}$ choose $n-1$ vectors ${\bf x}_{i_1},\ldots,{\bf x}_{i_{n-1}}$ such that $${\bf x}_{i_1},\ldots,{\bf x}_{i_{n-1}},{\bf y}_1\
     \textrm{is a basis}$$ in $\FFq2^n$. Since ${\bf x}_{i_1}^{\ast}B_0^{-1}{\bf x}_{i_1}={\bf e}_{i_1}^{\ast}P^{\ast}(PP^{\ast})^{-1}P{\bf e}_{i_1}=1$,  we can use Corollary~\ref{handbook} and choose $\alpha_1\in\{\mu_1,\mu_2\}$ such that
     \begin{align}
     \nonumber {\bf x}_{i_1}^{\ast}(B_0+\alpha_1 {\bf y}_1{\bf y}_1^{\ast})^{-1}{\bf x}_{i_1}&={\bf x}_{i_1}^{\ast}\left(B_0^{-1}-\frac{1}{\alpha_1^{-1}+{\bf y}_1^{\ast}B_0^{-1}{\bf y}_1} (B_0^{-1}{\bf y}_1)(B_0^{-1}{\bf y}_1)^{\ast}\right){\bf x}_{i_1}\\
     \label{eq29}&=1-\frac{{\bf x}_{i_1}^{\ast}B_0^{-1}{\bf y}_{1}\cdot \inv{{\bf x}_{i_1}^{\ast}B_0^{-1}{\bf y}_{1}}}{\alpha_1^{-1}+{\bf y}_1^{\ast}B_0^{-1}{\bf y}_1}\neq 0.
     \end{align}
     Let $$B_1:=B_0+\alpha_1 {\bf y}_1{\bf y}_1^{\ast}.$$ By (\ref{eq29}) and Corollary~\ref{handbook} we obtain $\eta_1\in\FF$ such that $$B_2:=B_1+(\eta_1-1){\bf x}_{i_1}{\bf x}_{i_1}^{\ast}=\eta_1 {\bf x}_{i_1}{\bf x}_{i_1}^{\ast}+\sum_{j\neq i_1}{\bf x}_{j}{\bf x}_{j}^{\ast}+ \alpha_1 {\bf y}_1{\bf y}_1^{\ast}$$
     has determinant $\lambda$. If $\eta_1\neq 0$, then we set
     $$B_3:=\eta_1 {\bf x}_{i_1}{\bf x}_{i_1}^{\ast}+\sum_{k=2}^{n-1}{\bf x}_{i_k}{\bf x}_{i_k}^{\ast}+ \alpha_1 {\bf y}_1{\bf y}_1^{\ast}.$$
     By Lemma~\ref{det-tenzor} we can replace $\eta_1$ by suitable nonzero $\nu_1$, such that
     $$B_4:=\nu_1 {\bf x}_{i_1}{\bf x}_{i_1}^{\ast}+\sum_{k=2}^{n-1}{\bf x}_{i_k}{\bf x}_{i_k}^{\ast}+ \alpha_1 {\bf y}_1{\bf y}_1^{\ast}$$
     has determinant $\lambda$. If $\eta_1=0$, we set $B_4=B_3:=B_2$. In both cases there are column vectors ${\bf z}_1,\ldots,{\bf z}_{n-1}$ such that $B_4=\sum_{i=1}^{n-1}{\bf z}_i{\bf z}_i^{\ast}+\alpha_1 {\bf y}_1{\bf y}_1^{\ast}$. Since $B_4$ is invertible, ${\bf z}_1,\ldots,{\bf z}_{n-1},{\bf y}_1$ form a basis in $\FFq2^n$. We now repeat the procedure, that is, we define $B_5,B_6,B_7,B_8$ similarly as we have defined $B_1,B_2,B_3,B_4$. More precisely, we first choose $n-2$ vectors ${\bf z}_{j_1},\ldots,{\bf z}_{j_{n-2}}$ in the set $\{{\bf z}_1,\ldots,{\bf z}_{n-1}\}$ such that ${\bf z}_{j_1},\ldots,{\bf z}_{j_{n-2}},{\bf y}_1,{\bf y}_2$ is a basis in $\FFq2^n$. Then we obtain a nonzero~$\alpha_2$ such that $B_5:=B_4+\alpha_2 {\bf y}_2{\bf y}_2^{\ast}$ is invertible with ${\bf z}_{j_1}^{\ast}(B_4+\alpha_2 {\bf y}_2{\bf y}_2^{\ast})^{-1}{\bf z}_{j_1}\neq 0$. Then we pick $\eta_2\in\FF$ such that $B_6:=\eta_2 {\bf z}_{j_1}{\bf z}_{j_1}^{\ast}+\sum_{k\neq j_1}{\bf z}_{k}{\bf z}_{k}^{\ast}+ \alpha_1 {\bf y}_1{\bf y}_1^{\ast}+ \alpha_2 {\bf y}_2{\bf y}_2^{\ast}$ has determinant $\lambda$. Matrices $B_7,B_8$ are defined similarly as $B_3,B_4$, so $B_8=\sum_{i=1}^{n-2}{\bf w}_i{\bf w}_i^{\ast}+\alpha_1 {\bf y}_1{\bf y}_1^{\ast}+\alpha_2 {\bf y}_2{\bf y}_2^{\ast}$ for suitable column vectors ${\bf w}_1,\ldots,{\bf w}_{n-2}$. Moreover $\det B_8=\lambda$. We proceed with this four step procedure till we construct matrix $B_{4t}=\sum_{i=1}^{n}\alpha_i {\bf y}_i{\bf y}_i^{\ast}$. Since $\det B_{4t}=\det A_1=\det A_2$, Lemma~\ref{det-tenzor} shows that $\prod_{i=1}^n\alpha_i=1$. We further define
     \begin{align*}
     B_{4t+1}&:={\bf y}_1{\bf y}_1^{\ast}+\sum_{i=2}^{n}\alpha_i {\bf y}_i{\bf y}_i^{\ast}\\
     B_{4t+2}&:={\bf y}_1{\bf y}_1^{\ast}+\alpha_1\alpha_2{\bf y}_2{\bf y}_2^{\ast}+\sum_{i=3}^{n}\alpha_i {\bf y}_i{\bf y}_i^{\ast}\\
     B_{4t+3}&:={\bf y}_1{\bf y}_1^{\ast}+{\bf y}_2{\bf y}_2^{\ast}+\sum_{i=3}^{n}\alpha_i {\bf y}_i{\bf y}_i^{\ast}\\
     B_{4t+4}&:={\bf y}_1{\bf y}_1^{\ast}+{\bf y}_2{\bf y}_2^{\ast}+\alpha_1\alpha_2\alpha_3{\bf y}_3{\bf y}_3^{\ast}+\sum_{i=4}^{n}\alpha_i {\bf y}_i{\bf y}_i^{\ast}\\
     &\vdots\\
     B_{2m}&={\bf y}_1{\bf y}_1^{\ast}+\ldots+{\bf y}_{n-1}{\bf y}_{n-1}^{\ast}+\alpha_1\alpha_2\cdots\alpha_n{\bf y}_n{\bf y}_n^{\ast}=A_2.
     \end{align*}
     By Lemma~\ref{det-tenzor}, all conditions in~(\ref{eq30}) are satisfied, so the proof ends.\qedhere
\end{myenumerate}
\end{proof}

\begin{cor}\label{povezanost}
Let $q\geq 4$. The graph $\hgl$ is connected.
\end{cor}
\begin{proof}
The determinant function on any $(q-1)$-clique attains all values from $\FF\backslash\{0\}$. Hence the connectedness of $\hgl$ follows from Proposition~\ref{lemma-povez-det}.
\end{proof}

\begin{lemma}\label{lemma-q-(q-1)}
Let $A\in \hgl$. The number of $q$-cliques and $(q-1)$-cliques that contain $A$ equal $|V_{\inv{A^{-1}}}|$ and $\frac{q^{2n}-1}{q^2-1}-|V_{\inv{A^{-1}}}|$ respectively. A vertex in $\hgl$ has $$|V_{\inv{A^{-1}}}|\cdot (q-1)+\left(\frac{q^{2n}-1}{q^2-1}-|V_{\inv{A^{-1}}}|\right)\cdot (q-2)$$
neighbors. If $q\geq 3$, then the number of $(q-1)$-cliques that contain $A$ is strictly bigger than the number of $q$-cliques that contain $A$.
\end{lemma}
\begin{proof}
The result follows from~(\ref{i5}), Corollary~\ref{handbook},~(\ref{i12}), and Theorem~\ref{kardinalnost}.
\end{proof}

\begin{lemma}\label{klika-kromaticno}
The core of $\hgl$ has chromatic number strictly bigger than~$q$. In particular, the core of $\hgl$ is not complete.
\end{lemma}
\begin{proof}
Let $\Gamma'$ be a core of $\hgl$. Assume that $\chi(\Gamma')\leq q$. If $I_{n-2}$ denotes the $(n-2)\times (n-2)$ identity matrix, then the map $A\mapsto A\oplus I_{n-2}$ is a homomorphism from $\hgldva$ to $\hgl$, so its composition with any homomorphism $\Phi: \hgl\to \Gamma'$ forms a homomorphism from $\hgldva$ to $\Gamma'$. Hence, $\chi(\hgldva)= q$ by Lemma~\ref{lema-kromaticno} and the fact that $\hgldva$ contains $q$-cliques. Recall from Section~2 that the eigenvalues of graph ${\cal H}_2(\FFq2)$, that is formed by all $2\times 2$ hermitian matrices,  equal
\begin{equation}\label{i9}
q^3-q^2+q-1\geq \underbrace{q-1\geq \ldots \geq q-1}_{q^4-k-1}\geq \underbrace{-q^2+q-1\geq \ldots \geq -q^2+q-1}_{k},
\end{equation}
where $k=q^3-q^2+q-1$ is the number of all rank--one matrices. If we delete them one by one, the interlacing~(\ref{i7}) implies that the induced subgraph formed by invertible matrices and the zero matrix has eigenvalues
\begin{equation}\label{i13}
\lambda\geq \underbrace{q-1\geq \ldots \geq q-1}_{q^4-2k-1}\geq \nu_1\geq \nu_2 \ldots \geq \nu_k,
\end{equation}
where $\nu_i\geq -q^2+q-1$ for all $i$. To obtain $\hgldva$ we still need to delete the zero matrix.
However, the zero matrix is not adjacent to rank-two matrices, so eigenvalues~(\ref{i13}) are precisely the eigenvalues of $\hgldva$ together with the value~$0$. Consequently, the eigenvalues of $\hgldva$ are
\begin{equation*}
\lambda\geq \underbrace{q-1\geq \ldots \geq q-1}_{q^4-2k-1}\geq \eta_1\geq\ldots \geq \eta_{k-1},
\end{equation*}
with unknown $\eta_1,\ldots,\eta_{k-1}\geq -q^2+q-1$, while (\ref{i11}), Lemma~\ref{lemma-q-(q-1)}, and Theorem~\ref{kardinalnost} imply that $\lambda=q^3-2q^2+2q-1$. By~(\ref{i10}),
\begin{equation}\label{i14}
\sum_{i=1}^{k-1}\eta_i=-(q^3-2q^2+2q-1)-(q-1)(q^4-2k-1)=-q^5+3q^4-5q^3+6q^2-5q+2.
\end{equation}
We separate two cases and note that $k-q\geq 1$.
\begin{myenumerate}{Case}
\item Assume that $\eta_{k-q}\geq -q^2+2q-1$. Then $\eta_1,\ldots,\eta_{k-q}\geq -q^2+2q-1$ and $\eta_{k-q+1},\ldots,\eta_{k-1}\geq -q^2+q-1$, so~(\ref{i14}) implies that
    \begin{align*}-q^5+3q^4-5q^3+6q^2-5q+2&\geq (k-q)(-q^2+2q-1)+(q-1)(-q^2+q-1)\\
    &=-q^5+3q^4-4q^3+4q^2-4q+2,
    \end{align*}
    which simplifies into $-q(q-1)^2\geq 0$, a contradiction.
\item Assume that $\eta_{k-q}<-q^2+2q-1$. In Theorem~\ref{haemers} applied at $\hgldva$ we have
$\lambda_2=\ldots=\lambda_{\chi(\hgldva)}=q-1$ and $\lambda_{t-\chi(\hgldva)+1}=\eta_{k-q}$. Consequently,
    $$0\leq (q-1)(q-1)+\eta_{k-q}<(q-1)(q-1)+(-q^2+2q-1)=0,$$ a contradiction. Hence, $\chi(\Gamma')>q$.\qedhere
\end{myenumerate}
\end{proof}

Next four lemmas study some transitivity properties of maps $A\mapsto PAP^{\ast}$.
\begin{lemma}\label{pomozna}
Suppose that nonzero vectors ${\bf x},{\bf y}\in\FFq2^n$ satisfy ${\bf x}^{\ast}{\bf x}=0={\bf y}^{\ast}{\bf y}$. Then there exists an invertible matrix $P$ such that $P^{\ast}P=I$ and $P{\bf x}={\bf y}$.
\end{lemma}
\begin{proof}
Pick $a\in\FFq2$ such that $\N(a)=-1$. It suffices to prove the conclusion for ${\bf x}=(1,a,0,\ldots,0)^{\tr}$. In fact, then there are matrices $P_1, P_2$ with $P_1(1,a,0,\ldots,0)^{\tr}={\bf y}$, $P_2(1,a,0,\ldots,0)^{\tr}={\bf x}$, and $P_i^{\ast}P_i=I$, so $P:=P_1P_2^{-1}$ solves the original problem.

So let ${\bf x}=(1,a,0,\ldots,0)^{\tr}$. We first show that there exists a vector ${\bf z}$ such that
\begin{equation}\label{eq24}
{\bf z}^{\ast}{\bf z}=1\quad \textrm{and} \quad {\bf z}^{\ast}{\bf y}=a.
\end{equation}
We consider three cases based on coordinates of ${\bf y}=(y_1,\ldots,y_n)^{\tr}$.
\begin{myenumerate}{Case}
\item Assume that $y_k=0$ for some $k$.

Then we pick an arbitrary vector ${\bf z}$ that satisfy ${\bf z}^{\ast}{\bf y}=a$ and change its $k$-th coordinate to achieve ${\bf z}^{\ast}{\bf z}=1$.

\item  Assume that  $y_k\neq 0$ for all $k$ and $\N(y_i)+\N(y_j)=0$ for some $i\neq j$.

Then we define ${\bf z}=(z_1,\ldots,z_n)^{\tr}$ such that $z_k=0$ for all $k\neq i,j$, coordinate~$z_j$ is chosen in a way to satisfy equation $ \Tr(z_j a \inv{y}_j)=\N(y_j)-1$, while $z_i=(\inv{a}-z_j\inv{y}_j)(\inv{y}_i)^{-1}$. A straightforward calculation shows that~(\ref{eq24}) is satisfied.

\item  Assume that  $y_k\neq 0$ for all $k$ and  $\N(y_i)+\N(y_j)\neq 0$ for all $i\neq j$.

Then $\frac{\N(y_2)}{\N(y_1)}+1\neq 0$ and the assumption ${\bf y}^{\ast}{\bf y}=0$ implies that $n\geq 3$. Denote ${\bf v}=(y_3,\ldots,y_n)^{\tr}$. Choose any ${\bf u}\in\FFq2^{n-2}$ such that ${\bf u}^{\ast}{\bf v}=a$. Pick $z_2$ such that $$\N(z_2)=\frac{1-{\bf u}^{\ast}{\bf u}}{\frac{\N(y_2)}{\N(y_1)}+1}\quad \textrm{and define}\quad z_1=-z_2\inv{\left(\frac{y_2}{y_1}\right)}.$$ A straightforward calculation shows that~(\ref{eq24}) is satisfied for ${\bf z}=(z_1,z_2,{\bf u}^{\tr})^{\tr}$.
\end{myenumerate}
So let ${\bf z}$ be as in~(\ref{eq24}). Then ${\bf x}_1:={\bf y}-a{\bf z}$ and ${\bf x}_2:={\bf z}$ are orthonormal. By Propositions~\ref{ortonormalnost}, we can enlarge  $\{{\bf x}_1,{\bf x}_2\}$ to an orthonormal basis $\{{\bf x}_1,\ldots,{\bf x}_n\}$. The matrix $P$, with ${\bf x}_i$ as the $i$-th column, has the required properties.
\end{proof}

\begin{lemma}\label{4x4}
Let $n\geq 4$. For $i=1,2$ assume that vectors ${\bf x}_i,{\bf y}_i\in\FFq2^n$ are linearly independent and satisfy $0={\bf x}_i^{\ast}{\bf x}_i={\bf y}_i^{\ast}{\bf y}_i={\bf x}_i^{\ast}{\bf y}_i$. Then there exist a matrix $P$ and a nonzero $b\in\FFq2$ such that $P^{\ast}P=I$, $P{\bf x}_1={\bf x}_2$, and $P{\bf y}_1=b {\bf y}_2$.
\end{lemma}
\begin{proof}
Fix $a\in\FFq2$ such that $\N(a)=-1$ and let ${\bf e}_i$ be the $i$-th standard vector. We split the proof in four steps.
\begin{myenumerate}{Step}
\item The conclusion holds if ${\bf x}_1:={\bf e}_1+a{\bf e}_2=:{\bf x}_2$ and ${\bf y}_1:={\bf e}_3+a{\bf e}_4$.

Let ${\bf y}_2=(u_1,u_2,{\bf u}^{\tr})^{\tr}$, where ${\bf u}\in\FFq2^{n-2}$. The assumptions ${\bf x}_2^{\ast}{\bf y}_2=0={\bf y}_2^{\ast}{\bf y}_2$ imply that $u_1=-\inv{a}u_2$ and ${\bf u}^{\ast}{\bf u}=0$.

Assume first that $u_2=0$. Then $u_1=0$. By Lemma~\ref{pomozna} there exists an $(n-2)\times (n-2)$ matrix $R$ such that $R^{\ast}R=I_{n-2}$ and $R(1,a,0\ldots,0)^{\tr}={\bf u}$. Consequently, $P:=I_{2}\oplus R$ and $b:=1$ have the required properties.

Assume now that $u_2\neq 0$. By Lemma~\ref{pomozna} there is an $(n-2)\times (n-2)$ matrix $R$ such that $R^{\ast}R=I_{n-2}$ and $R(\inv{a},1,0\ldots,0)^{\tr}=u_2^{-1}{\bf u}$.
Choose $d$ such that $\Tr(d)=1$. A straightforward calculation shows that
$$Q:=\begin{bmatrix}
2-d& \inv{ad}&-\inv{a}&0\\
a\inv{d}& d& 1&0\\
0& 0&0&-\inv{a}^2\\
a& -1&1&0
\end{bmatrix}\oplus I_{n-4}$$
satisfies $Q^{\ast}Q=I$. Moreover, $P:=(I_{2}\oplus R)Q$ and $b:=u_2^{-1}$ solve our problem.

\item The conclusion holds if ${\bf x}_1:={\bf e}_1+a{\bf e}_2=:{\bf x}_2$.

By Step~1, there are matrices $P_1$, $P_2$ and nonzero scalars $b_1, b_2$ such that $P_1{\bf x}_1={\bf x}_2$, $P_1({\bf e}_3+a{\bf e}_4)=b_1{\bf y}_1$,
$P_2{\bf x}_1={\bf x}_2$, $P_2({\bf e}_3+a{\bf e}_4)=b_2{\bf y}_2$,  and $P_i^{\ast}P_i=I$. Then, $P:=P_2P_1^{-1}$ and $b:=b_2b_1^{-1}$ have the required properties.

\item The conclusion holds if ${\bf x}_1={\bf x}_2$.

By Lemma~\ref{pomozna}, there exists a matrix $Q$ with $Q^{\ast} Q=I$ and $Q{\bf x}_1={\bf e}_1+a{\bf e}_2$. Consequently $0=(Q{\bf x}_1)^{\ast}(Q{\bf y}_1)=(Q{\bf x}_1)^{\ast}(Q{\bf y}_2)=(Q{\bf y}_1)^{\ast}(Q{\bf y}_1)=(Q{\bf y}_2)^{\ast}(Q{\bf y}_2)$. By Step~2, there exist $R$ and $b$ with $R(Q{\bf x}_1)=(Q{\bf x}_1)$, $R(Q{\bf y}_1)=b(Q{\bf y}_2)$, and $R^{\ast}R=I$. If $P:=Q^{\ast}RQ$, then $P^{\ast}P=I$, $P{\bf x}_1={\bf x}_2$, and $P{\bf y}_1=b{\bf y}_2$.

\item The conclusion holds whenever the assumptions in lemma are satisfied.

By Lemma~\ref{pomozna}, there is a matrix $Q$ with $Q^{\ast} Q=I$ and $Q{\bf x}_1={\bf x}_2$. Consequently $0={\bf x}_2^{\ast}(Q{\bf y}_1)=(Q{\bf y}_1)^{\ast}(Q{\bf y}_1)$. By Step~3, there exist $R$ and $b$ with $R{\bf x}_2={\bf x}_2$, $R(Q{\bf y}_1)=b{\bf y}_2$, and $R^{\ast}R=I$. Matrix $P:=RQ$ has the required properties.
\qedhere
\end{myenumerate}
\end{proof}

\begin{lemma}\label{nova}
For $i=1,2$ assume that vectors ${\bf x}_i,{\bf y}_i\in\FFq2^n$ are linearly independent and satisfy $0={\bf x}_i^{\ast}{\bf x}_i={\bf y}_i^{\ast}{\bf y}_i$ and $0\neq {\bf x}_i^{\ast}{\bf y}_i$. Then there exist a matrix $P$ and a nonzero $b\in\FFq2$ such that $P^{\ast}P=I$, $P{\bf x}_1={\bf x}_2$, and $P{\bf y}_1=b {\bf y}_2$.
\end{lemma}
\begin{proof}
Fix $i\in\{1,2\}$. By~(\ref{i3}), there is $a\in \FFq2\backslash\{1\}$ such that $\N(a)=-1$. According on whether $q$ is odd or even define ${\bf u}, {\bf v}, {\bf t}_1, {\bf t}_2$ as follows:
$$\begin{tabular}{|l|l|}
  \hline
   $q$ odd & $q$ even \\
  \hline
  ${\bf u}:={\bf e}_1+a{\bf e}_2$& ${\bf u}:={\bf e}_1+a{\bf e}_2$\\
${\bf v}:=\inv{a}{\bf e}_1+{\bf e}_2$& ${\bf v}:=a{\bf e}_1+{\bf e}_2$\\
${\bf t}_1:=\frac{1}{2}\cdot  {\bf x}_i+\frac{1}{{\bf x}_i^{\ast}{\bf y}_i}\cdot {\bf y}_i$& ${\bf t}_1:=\frac{\Tr(a)}{1+a^2}\cdot {\bf x}_i+\frac{a}{(1+a^2){\bf x}_i^{\ast}{\bf y}_i}\cdot  {\bf y}_i$\\
${\bf t}_2:=\frac{\inv{a}}{2} \cdot {\bf x}_i-\frac{\inv{a}}{{\bf x}_i^{\ast}{\bf y}_i}\cdot  {\bf y}_i$& ${\bf t}_2:=\frac{a\Tr(a)}{1+a^2}\cdot  {\bf x}_i+\frac{1}{(1+a^2){\bf x}_i^{\ast}{\bf y}_i} \cdot {\bf y}_i$\\
  \hline
\end{tabular}$$
Vectors ${\bf t}_1$ and ${\bf t}_2$ are orthonormal  in both cases. By Proposition~\ref{ortonormalnost}, we can enlarge $\{{\bf t}_1,{\bf t}_2\}$ to an orthonormal basis $\{{\bf t}_1,\ldots,{\bf t}_n\}$. Let $R_i$ be the matrix with ${\bf t}_j$ as the $j$-th column. Then $R_i^{\ast}R_i=I$ and obviously $R_i R_i^{\ast}=I$ as well. Moreover,
if $q$ is odd then
$$R_i {\bf u}=\frac{2}{{\bf x}_i^{\ast}{\bf y}_i}\cdot {\bf y}_i\quad \textrm{and}\quad R_i {\bf v}=\inv{a}\cdot {\bf x}_i,$$
and if $q$ is even, then
$$R_i {\bf u}=\Tr(a)\cdot {\bf x}_i\quad \textrm{and}\quad R_i {\bf v}=\frac{1}{{\bf x}_i^{\ast}{\bf y}_i}\cdot {\bf y}_i,$$
so in both cases $P:=R_2R_1^{-1}$ and $b:=\frac{{\bf x}_1^{\ast}{\bf y}_1}{{\bf x}_2^{\ast}{\bf y}_2}$ possess the required properties.
\end{proof}

Below we use $l_{\bf x}^{A}$ to denote the maximal clique in $\hgl$ formed by  $0\neq {\bf x}\in\FFq2^n$ and $A\in\hgl$, that is, $l_{\bf x}^{A}=\{A+\lambda {\bf x}{\bf x}^{\ast}\ :\ \lambda \in \FF, \lambda{\bf x}^{\ast}A^{-1}{\bf x}\neq -1\}$. So $l_{\bf x}^A$ is a $q$-clique if ${\bf x}^{\ast}A^{-1}{\bf x}=0$ and a $(q-1)$-clique if ${\bf x}^{\ast}A^{-1}{\bf x}\neq 0$. We say that $l_{\bf x}^A$ is a \emph{clique of $A$}. If $A=I$,  we simply write $l_{\bf x}$. We say that  $l_{\bf x}^A$ and $l_{\bf y}^A$ are \emph{$A$-orthogonal} if ${\bf x}^{\ast}A^{-1}{\bf y}=0$. If $A=I$, we call them \emph{orthogonal}.

\begin{lemma}\label{lemma-pari2-simetrija}
For $i\in\{1,2\}$ let $A_i\in \hgl$ and suppose that $l_{{\bf x}_i}^{A_i}$ and $l_{{\bf y}_i}^{A_i}$ are two distinct $q$-cliques.
\begin{enumerate}
\item If $n\in\{2,3\}$, then $l_{{\bf x}_i}^{A_i}$ and $l_{{\bf y}_i}^{A_i}$ are $A_i$-non-orthogonal.
\item If $l_{{\bf x}_i}^{A_i}$ and $l_{{\bf y}_i}^{A_i}$ are either $A_i$-non-orthogonal for both $i\in\{1,2\}$
 or $A_i$-orthogonal for both $i\in\{1,2\}$, then there exists an invertible matrix $P$ such that $PA_1 P^{\ast}=A_2$, $Pl_{{\bf x}_1}^{A_1}P^{\ast}=l_{{\bf x}_2}^{A_2}$, and $Pl_{{\bf y}_1}^{A_1}P^{\ast}=l_{{\bf y}_2}^{A_2}$.
\end{enumerate}
\end{lemma}
\begin{proof}
\begin{enumerate}
\item Suppose that  ${\bf x}_i^{\ast}A_i^{-1}{\bf y}_i= 0$ for some $i\in\{1,2\}$. Since ${\bf x}_i^{\ast}A_i^{-1}{\bf x}_i=0={\bf y}_i^{\ast}A_i^{-1}{\bf y}_i$,  $V_{\inv{A_i^{-1}}}$ contains the two dimensional space spanned by ${\bf x}_i$ and ${\bf y}_i$, so $n\geq 4$ by Theorem~\ref{bose_podprostor}.
\item For $i=1,2$ pick an invertible matrix $Q_i$ such that $A_i=Q_iQ_i^{\ast}$ and denote ${\bf z}_i:=Q_i^{-1}{\bf x}_i$, ${\bf w}_i:=Q_i^{-1}{\bf y}_i$. Since $Q_i^{-1}l_{{\bf x}_i}^{A_i}(Q_i^{-1})^{\ast}=l_{{\bf z}_i}$ and $Q_i^{-1}l_{{\bf y}_i}^{A_i}(Q_i^{-1})^{\ast}=l_{{\bf w}_i}$ are $q$-cliques, we deduce that
${\bf z}_i^{\ast}{\bf z}_i=0={\bf w}_i^{\ast}{\bf w}_i$.

Assume first that ${\bf x}_1^{\ast}A_1^{-1}{\bf y}_1= 0={\bf x}_2^{\ast}A_2^{-1}{\bf y}_2$. Then $n\geq 4$ and ${\bf z}_i^{\ast}{\bf w}_i=0$. By Lemma~\ref{4x4} there is a matrix $R$ and $b\neq 0$ such that $R^{\ast}R=I$, $R{\bf z}_1={\bf z}_2$, and $R{\bf w}_1=b {\bf w}_2$, so $P:=Q_2RQ_1^{-1}$ has the required properties.

If ${\bf x}_1^{\ast}A_1^{-1}{\bf y}_1\neq 0\neq {\bf x}_2^{\ast}A_2^{-1}{\bf y}_2$, then ${\bf z}_i^{\ast}{\bf w}_i\neq 0$, and we can repeat the proof above by replacing Lemma~\ref{4x4} with Lemma~\ref{nova}.\qedhere
\end{enumerate}
\end{proof}

In the proof of next lemma we tacitly rely on the following fact: an adjacency preserving map, i.e. an endomorphism of graph $\hgl$, maps any maximal clique injectively into some maximal clique.

\begin{lemma}\label{lemma-main}
Assume that $q\geq 4$ and ${\bf x}_1, {\bf x}_2\in \FF_{q^2}^n$ satisfy ${\bf x}_1^{\ast}{\bf x}_1=0$ and ${\bf x}_1^{\ast}{\bf x}_2\neq 0$.  If
$\Phi: \hgl\to\hgl$ preserves adjacency and $\Phi(l_{{\bf x}_2})\subseteq \Phi(l_{{\bf x}_1})$, then $\Phi(l_{a_1 {\bf x}_1+a_2{\bf x}_2})\subseteq \Phi(l_{{\bf x}_1})$ for all $a_1,a_2\in \FF_{q^2}$ with $(a_1,a_2)\neq (0,0)$.
\end{lemma}

\begin{proof}
There is nothing to prove if $a_1=0$ or $a_2=0$, so assume that both are nonzero.
It follows from the assumptions that ${\bf x}_1, {\bf x}_2$ are linearly independent and $I+\lambda {\bf x}_1{\bf x}_1^{\ast}$ is invertible with inverse $I-\lambda {\bf x}_1{\bf x}_1^{\ast}$ for all $\lambda\in \FF$.
We separate four cases.
\begin{myenumerate}{Case}
\item Let ${\bf x}_2^{\ast}{\bf x}_2=0$ and $(a_1 {\bf x}_1+a_2{\bf x}_2)^{\ast}(a_1 {\bf x}_1+a_2{\bf x}_2)\neq 0$.\smallskip

       Choose
     $\lambda_0:=\frac{(a_1 {\bf x}_1+a_2{\bf x}_2)^{\ast}(a_1 {\bf x}_1+a_2{\bf x}_2)}{a_2\inv{a}_2{\bf x}_1^{\ast}{\bf x}_2{\bf x}_2^{\ast}{\bf x}_1}$. Then
    \begin{equation*}
     (a_1 {\bf x}_1+a_2{\bf x}_2)^{\ast} (I+\lambda_0 {\bf x}_1{\bf x}_1^{\ast})^{-1} (a_1 {\bf x}_1+a_2{\bf x}_2)=0
    \end{equation*}
    and $I+\lambda_0 {\bf x}_1{\bf x}_1^{\ast}+\mu (a_1 {\bf x}_1+a_2{\bf x}_2)(a_1 {\bf x}_1+a_2{\bf x}_2)^{\ast}$ is  invertible for all $\mu \in\FF$ by Corollary~\ref{handbook}. Since ${\bf x}_2^{\ast}{\bf x}_2=0$, $|l_{{\bf x}_2}|=q$, so $\Phi(l_{{\bf x}_2})= \Phi(l_{{\bf x}_1})$, that is,
     $\Phi(I+\nu {\bf x}_2{\bf x}_2^{\ast})=\Phi(I+f(\nu) {\bf x}_1{\bf x}_1^{\ast})$ for all $\nu\in\FF$, where $f$ is a bijection on $\FF$ with $f(0)=0$. Since $q\geq 4$, there exists a nonzero $\nu_0$ such that $f(\nu_0)\neq \lambda_0$ and $\lambda_0 \N(a_2)-\nu_0 \N(a_1)\neq 0$. Choose $\mu_0:=\frac{\lambda_0\nu_0}{\lambda_0 \N(a_2)-\nu_0 \N(a_1)}$. Then $I+\lambda_0 {\bf x}_1{\bf x}_1^{\ast}+\mu_0 (a_1 {\bf x}_1+a_2{\bf x}_2)(a_1 {\bf x}_1+a_2{\bf x}_2)^{\ast}$ and $I+\nu_0 {\bf x}_2{\bf x}_2^{\ast}$ are adjacent. Hence, $\Phi(I+\lambda_0 {\bf x}_1{\bf x}_1^{\ast})$, $\Phi(I+\lambda_0 {\bf x}_1{\bf x}_1^{\ast}+\mu_0 (a_1 {\bf x}_1+a_2{\bf x}_2)(a_1 {\bf x}_1+a_2{\bf x}_2)^{\ast})$, and $\Phi(I+f(\nu_0) {\bf x}_1{\bf x}_1^{\ast})$ are pairwise adjacent. Since the maximal clique containing both $\Phi(I+\lambda_0 {\bf x}_1{\bf x}_1^{\ast})$ and $\Phi(I+f(\nu_0) {\bf x}_1{\bf x}_1^{\ast})$ is unique by (\ref{i5}) and equals $\Phi(l_{{\bf x}_1})$, it follows that $\Phi(I+\lambda_0 {\bf x}_1{\bf x}_1^{\ast}+\mu_0 (a_1 {\bf x}_1+a_2{\bf x}_2)(a_1 {\bf x}_1+a_2{\bf x}_2)^{\ast})\in\Phi(l_{{\bf x}_1})$. Since $\Phi(I+\lambda_0 {\bf x}_1{\bf x}_1^{\ast})\in\Phi(l_{{\bf x}_1})$ as well, we deduce that $\Phi(I+\lambda_0 {\bf x}_1{\bf x}_1^{\ast}+\mu (a_1 {\bf x}_1+a_2{\bf x}_2)(a_1 {\bf x}_1+a_2{\bf x}_2)^{\ast})\in\Phi(l_{{\bf x}_1})$, i.e.,
     $\Phi(I+\lambda_0 {\bf x}_1{\bf x}_1^{\ast}+\mu (a_1 {\bf x}_1+a_2{\bf x}_2)(a_1 {\bf x}_1+a_2{\bf x}_2)^{\ast})=\Phi(I+g(\mu) {\bf x}_1{\bf x}_1^{\ast})$ for all $\mu\in\FF$, where $g$ is a bijection on $\FF$. Since $q\geq 4$, there exists a nonzero $\mu_1$ such that $g(\mu_1)\neq 0$ and $I+\mu_1 (a_1 {\bf x}_1+a_2{\bf x}_2)(a_1 {\bf x}_1+a_2{\bf x}_2)^{\ast}$ is invertible. Since $\Phi(I+\mu_1 (a_1 {\bf x}_1+a_2{\bf x}_2)(a_1 {\bf x}_1+a_2{\bf x}_2)^{\ast})$, $\Phi(I)$, and $\Phi(I+g(\mu_1) {\bf x}_1{\bf x}_1^{\ast})$ are pairwise adjacent, it follows that $\Phi(I+\mu_1 (a_1 {\bf x}_1+a_2{\bf x}_2)(a_1 {\bf x}_1+a_2{\bf x}_2)^{\ast})\in \Phi(l_{{\bf x}_1})$. Consequently $\Phi(l_{a_1 {\bf x}_1+a_2{\bf x}_2})\subseteq \Phi(l_{{\bf x}_1})$.

\item Let ${\bf x}_2^{\ast}{\bf x}_2=0$ and $(a_1 {\bf x}_1+a_2{\bf x}_2)^{\ast}(a_1 {\bf x}_1+a_2{\bf x}_2)=0$.\smallskip

     As in Case~1, $\Phi(I+\nu {\bf x}_2{\bf x}_2^{\ast})=\Phi(I+f(\nu) {\bf x}_1{\bf x}_1^{\ast})$ for all $\nu\in\FF$, where $f$ is a bijection on $\FF$ with $f(0)=0$. Pick an arbitrary nonzero $\lambda\in\FF$. Since $q\geq 4$, there exists a nonzero $\mu_0\in\FF$ such that $I+\lambda {\bf x}_1{\bf x}_1^{\ast}+\mu_0 {\bf x}_2{\bf x}_2^{\ast}$ is invertible and $f(\mu_0)\neq \lambda$. Since matrices $\Phi(I+\lambda {\bf x}_1{\bf x}_1^{\ast})$, $\Phi(I+\lambda {\bf x}_1{\bf x}_1^{\ast}+\mu_0 {\bf x}_2{\bf x}_2^{\ast})$, and $\Phi(I+f(\mu_0) {\bf x}_1{\bf x}_1^{\ast})$ are pairwise adjacent, it follows that $\Phi(I+\lambda {\bf x}_1{\bf x}_1^{\ast}+\mu_0 {\bf x}_2{\bf x}_2^{\ast})\in \Phi(l_{{\bf x}_1})$. Consequently,
    \begin{equation}\label{eq16}
    \Phi(I+\lambda {\bf x}_1{\bf x}_1^{\ast}+\mu {\bf x}_2{\bf x}_2^{\ast})\in \Phi(l_{{\bf x}_1})
    \end{equation}
    for all $\lambda,\mu\in\FF$ such that $I+\lambda {\bf x}_1{\bf x}_1^{\ast}+\mu {\bf x}_2{\bf x}_2^{\ast}$ is invertible.

    If $q\geq 5$, choose $0\neq\lambda_0\in \FF$ arbitrarily.
    Then there are two distinct nonzero $\mu_1,\mu_2\in\FF$ such that
    \begin{equation}\label{eq15}
    I+\lambda_0 {\bf x}_1{\bf x}_1^{\ast}+\mu_j {\bf x}_2{\bf x}_2^{\ast}\ \textrm{is invertible and}\ \lambda_0 \N(a_2)+\mu_j \N(a_1)\neq 0\quad (j=1,2).
    \end{equation}
    If $q=4$, then any element in $\FF$ is a square, so there exists a nonzero $\lambda_0$ such that $\lambda_0^2 \N(a_2)+\frac{\N(a_1)}{{\bf x}_1^{\ast}{\bf x}_2{\bf x}_2^{\ast}{\bf x}_1}=0$, that is, $\lambda_0 \N(a_2)+\frac{1}{\lambda_0{\bf x}_1^{\ast}{\bf x}_2{\bf x}_2^{\ast}{\bf x}_1} \N(a_1)=0$. By Corollary~\ref{handbook}, $I+\lambda_0 {\bf x}_1{\bf x}_1^{\ast}+\mu {\bf x}_2{\bf x}_2^{\ast}$ is singular precisely for $\mu=\frac{1}{\lambda_0{\bf x}_1^{\ast}{\bf x}_2{\bf x}_2^{\ast}{\bf x}_1}$, so (\ref{eq15}) is still satisfied for appropriate $\mu_1$ and $\mu_2$.

    Since~(\ref{eq16}) holds, we can pick $j\in \{1,2\}$ such that $\Phi(I+\lambda_0 {\bf x}_1{\bf x}_1^{\ast}+\mu_j {\bf x}_2{\bf x}_2^{\ast})\neq \Phi(I)$. Let $\eta_j:=\frac{\lambda_0\mu_j}{\lambda_0 \N(a_2)+\mu_j \N(a_1)}$.
    Since $$\rank\big(\eta_j (a_1 {\bf x}_1+a_2{\bf x}_2)(a_1 {\bf x}_1+a_2{\bf x}_2)^{\ast}-\lambda_0 {\bf x}_1{\bf x}_1^{\ast}-\mu_j {\bf x}_2{\bf x}_2^{\ast}\big)=1,$$ it follows that matrices $\Phi(I+\eta_j (a_1 {\bf x}_1+a_2{\bf x}_2)(a_1 {\bf x}_1+a_2{\bf x}_2)^{\ast})$, $\Phi(I)$, and $\Phi(I+\lambda_0 {\bf x}_1{\bf x}_1^{\ast}+\mu_j {\bf x}_2{\bf x}_2^{\ast})$ are pairwise adjacent, so $\Phi(I+\eta_j (a_1 {\bf x}_1+a_2{\bf x}_2)(a_1 {\bf x}_1+a_2{\bf x}_2)^{\ast})\in \Phi(l_{{\bf x}_1})$ and consequently $\Phi(l_{a_1 {\bf x}_1+a_2{\bf x}_2})\subseteq \Phi(l_{{\bf x}_1})$.

    \item Let ${\bf x}_2^{\ast}{\bf x}_2\neq 0$ and either $(a_1 {\bf x}_1+a_2{\bf x}_2)^{\ast}(a_1 {\bf x}_1+a_2{\bf x}_2)=0$ or $q\geq 5$.\smallskip

        Choose $\lambda_0:=\frac{{\bf x}_2^{\ast}{\bf x}_2}{{\bf x}_1^{\ast}{\bf x}_2{\bf x}_2^{\ast}{\bf x}_1}$. By Corollary~\ref{handbook}, $I+\lambda_0 {\bf x}_1{\bf x}_1^{\ast}+\mu {\bf x}_2{\bf x}_2^{\ast}$ is invertible for all $\mu$. Similarly as (\ref{eq16}) in Case~2 we see that $\Phi(I+\lambda_0 {\bf x}_1{\bf x}_1^{\ast}+\mu {\bf x}_2{\bf x}_2^{\ast})\in \Phi(l_{{\bf x}_1})$ for all $\mu$. So there exists $\mu_0\neq 0$ such that $\Phi(I+\lambda_0 {\bf x}_1{\bf x}_1^{\ast}+\mu_0 {\bf x}_2{\bf x}_2^{\ast})\neq \Phi(I)$ and $\lambda_0 \N(a_2)+\mu_0 \N(a_1)\neq 0$. In fact, if $q\geq 5$, then there are two such scalars. So if either $(a_1 {\bf x}_1+a_2{\bf x}_2)^{\ast}(a_1 {\bf x}_1+a_2{\bf x}_2)=0$ or $q\geq 5$, we can choose $\mu_0$, such that $I+\eta_0(a_1 {\bf x}_1+a_2{\bf x}_2)(a_1 {\bf x}_1+a_2{\bf x}_2)^{\ast}$ is invertible for
        $\eta_0=\frac{\lambda_0\mu_0}{\lambda_0 \N(a_2)+\mu_0 \N(a_1)}$. We continue similarly as in Case~2.

    \item Let ${\bf x}_2^{\ast}{\bf x}_2\neq 0$,
    $(a_1 {\bf x}_1+a_2{\bf x}_2)^{\ast}(a_1 {\bf x}_1+a_2{\bf x}_2)\neq 0$, and $q=4$.\smallskip

        Fix a nonzero $a_2$. From~(\ref{i4}) we know that the equation
        \begin{equation}\label{eq17}
        \Tr(t)^2+\Tr(t)+\N(t)=0
        \end{equation}
        has four nonzero solutions $t_1,t_2,t_3,t_4\in \FF_{4^2}$.  We will first consider the case
        \begin{equation}\label{eq21}
         a_1\in\left\{a_2\frac{{\bf x}_2^{\ast}{\bf x}_2}{{\bf x}_2^{\ast}{\bf x}_1}t_j\ :\ j\in\{1,2,3,4\}\right\}.
        \end{equation}
        Since the characteristic of $\FF_{4^2}$ is two, (\ref{eq17}) implies that
        \begin{equation}\label{eq18}
        \Tr(t_j)+1\notin \{0,1\},
        \end{equation}
        so indeed
        \begin{equation}\label{eq22}
        (a_1 {\bf x}_1+a_2{\bf x}_2)^{\ast}(a_1 {\bf x}_1+a_2{\bf x}_2)=\N(a_2){\bf x}_2^{\ast}{\bf x}_2\big(\Tr(t_j)+1\big)\neq 0.
        \end{equation}
        Recall that $\FF=\{0,1,\imath,\imath^2\}$, where $\imath\in\FF_{4^2}$ satisfies $\imath^2=\imath+1$.
        Let $$\eta=\frac{\imath}{(a_1 {\bf x}_1+a_2{\bf x}_2)^{\ast}(a_1 {\bf x}_1+a_2{\bf x}_2)}.$$ By Corollary~\ref{handbook},
        \begin{align*}
        M_1&:=I+\eta \imath^2 (a_1 {\bf x}_1+a_2{\bf x}_2)(a_1 {\bf x}_1+a_2{\bf x}_2)^{\ast}\ \textrm{is singular,}\\
        M_2&:=I+\eta \imath (a_1 {\bf x}_1+a_2{\bf x}_2)(a_1 {\bf x}_1+a_2{\bf x}_2)^{\ast}\ \textrm{is invertible,}\\
        M_3&:=I+\eta  (a_1 {\bf x}_1+a_2{\bf x}_2)(a_1 {\bf x}_1+a_2{\bf x}_2)^{\ast}\ \textrm{is invertible.}
        \end{align*}
        Let $\mu:=\eta \imath^2 \N(a_2)$. It follows from (\ref{eq22}), (\ref{eq18}) that $\mu{\bf x}_2^{\ast}{\bf x}_2=\frac{1}{\Tr(t_j)+1}\neq -1$, so
        $$I+\mu{\bf x}_2{\bf x}_2^{\ast}\ \textrm{is invertible}$$
        by Corollary~\ref{handbook}.
        Corollary~\ref{handbook}, a short calculation, and (\ref{eq17}) imply that
        \begin{align*}
        \eta \imath^2 \N(a_1){\bf x}_1^{\ast}(I+\mu{\bf x}_2{\bf x}_2^{\ast})^{-1}{\bf x}_1&=\eta \imath^2 \N(a_1){\bf x}_1^{\ast}\left(I-\frac{\mu}{1+\mu{\bf x}_2^{\ast}{\bf x}_2}{\bf x}_2{\bf x}_2^{\ast}\right){\bf x}_1\\
        &=\frac{\N(t_j)}{\Tr(t_j)^2+\Tr(t_j)}=-1,
        \end{align*}
        so, by Corollary~\ref{handbook},
        \begin{align*}
        N_1&:=I+\mu{\bf x}_2{\bf x}_2^{\ast}+\eta \imath^2 \N(a_1){\bf x}_1{\bf x}_1^{\ast}\ \textrm{is singular,}\\
        N_2&:=I+\mu{\bf x}_2{\bf x}_2^{\ast}+\eta \imath\N(a_1){\bf x}_1{\bf x}_1^{\ast}\ \textrm{is invertible,}\\
        N_3&:=I+\mu{\bf x}_2{\bf x}_2^{\ast}+\eta \N(a_1){\bf x}_1{\bf x}_1^{\ast}\ \textrm{is invertible.}
        \end{align*}
        Moreover,
        \begin{equation}\label{eq23}
        \rank(N_2-M_3)=1=\rank(N_3-M_2).
        \end{equation}
        From the assumptions it follows that $\Phi(I+\mu{\bf x}_2{\bf x}_2^{\ast})=\Phi(I+\lambda{\bf x}_1{\bf x}_1^{\ast})$ for some nonzero $\lambda$. Pick $\nu\in\{\eta \N(a_1),\eta \imath \N(a_1)\}$ such that $\nu\neq \lambda$. Then $\Phi(I+\lambda{\bf x}_1{\bf x}_1^{\ast})$, $\Phi(I+\mu{\bf x}_2{\bf x}_2^{\ast}+\nu{\bf x}_1{\bf x}_1^{\ast})$, and $\Phi(I+\nu{\bf x}_1{\bf x}_1^{\ast})$ are all adjacent, so $\Phi(I+\mu{\bf x}_2{\bf x}_2^{\ast}+\nu{\bf x}_1{\bf x}_1^{\ast})\in \Phi(l_{{\bf x}_1})$. Consequently, $\Phi(N_2),\Phi(N_3)\in \Phi(l_{{\bf x}_1})$. Pick $N\in\{N_2,N_3\}$ such that $\Phi(N)\neq \Phi(I)$. By (\ref{eq23}), there is $M\in\{M_2,M_3\}$ such that $M$ and $N$ are adjacent. Since $\Phi(M)$, $\Phi(N)$, $\Phi(I)$ are all adjacent, we deduce that $\Phi(M)\in \Phi(l_{{\bf x}_1})$. Consequently $\Phi(l_{a_1 {\bf x}_1+a_2{\bf x}_2})\subseteq \Phi(l_{{\bf x}_1})$ for all scalars $a_1$ from the set~(\ref{eq21}).

        For a fixed $a_1'$ from~(\ref{eq21}), denote ${\bf y}_2=a_1' {\bf x}_1+a_2{\bf x}_2$. Since ${\bf x}_1^{\ast}{\bf y}_2\neq 0$ and
        ${\bf y}_2^{\ast}{\bf y}_2\neq 0$ by (\ref{eq22}), we can use the procedure above to see that $\Phi(l_{b_1' {\bf x}_1+{\bf y}_2})\subseteq \Phi(l_{{\bf x}_1})$ for all
        \begin{equation*}
         b_1'\in\left\{\frac{{\bf y}_2^{\ast}{\bf y}_2}{{\bf y}_2^{\ast}{\bf x}_1}t_k\ :\ k\in\{1,2,3,4\}\right\}.
        \end{equation*}
        Now, $b_1' {\bf x}_1+{\bf y}_2=(a_1'+b_1') {\bf x}_1+a_2{\bf x}_2$, where
        $$a_1'+b_1'=a_2\frac{{\bf x}_2^{\ast}{\bf x}_2}{{\bf x}_2^{\ast}{\bf x}_1}t_j+\frac{{\bf y}_2^{\ast}{\bf y}_2}{{\bf y}_2^{\ast}{\bf x}_1}t_k=a_2\frac{{\bf x}_2^{\ast}{\bf x}_2}{{\bf x}_2^{\ast}{\bf x}_1}\Big(t_j+\big(\Tr(t_j)+1\big)t_k\Big).$$ Hence, we have proved that $\Phi(l_{a_1 {\bf x}_1+a_2{\bf x}_2})\subseteq \Phi(l_{{\bf x}_1})$ for all
        \begin{equation}\label{eq19}
        a_1\in\left\{a_2\frac{{\bf x}_2^{\ast}{\bf x}_2}{{\bf x}_2^{\ast}{\bf x}_1}\Big(t_j+\big(\Tr(t_j)+1\big)t_k\Big)\ :\ j,k\in\{1,2,3,4\}\right\}.
        \end{equation}
        For such scalars we have $$(a_1 {\bf x}_1+a_2{\bf x}_2)^{\ast}(a_1 {\bf x}_1+a_2{\bf x}_2)=\N(a_2){\bf x}_2^{\ast}{\bf x}_2\big(\Tr(t_j)+1\big)\big(\Tr(t_k)+1\big)\neq 0$$ by~(\ref{eq18}).
        By~(\ref{i44}), the set in~(\ref{eq19}) consists of 11 nonzero elements. On the contrary, given a nonzero $a_2$, (\ref{i1}) shows that there are precisely $16-1-4=11$ nonzero scalars $a_1\in \FF_{4^2}$ satisfying $$(a_1 {\bf x}_1+a_2{\bf x}_2)^{\ast}(a_1 {\bf x}_1+a_2{\bf x}_2)=\Tr(a_1\inv{a}_2 {\bf x}_2^{\ast}{\bf x}_1)+\N(a_2){\bf x}_2^{\ast}{\bf x}_2\neq 0,$$ so we proved the conclusion for all $a_1$ that fits the assumptions of Case~4.\qedhere
\end{myenumerate}
\end{proof}

We are now ready to prove the main result of this paper.

\begin{proof}[Proof of Theorem~\ref{glavni}] Let $\Gamma$ be a core of  $\hgl$ and let $\Phi$ be any retraction of $\hgl$ onto $\Gamma$. We separate three cases. In first we  obtain the desired $\Gamma=\hgl$, while the other two lead to a contradiction.
\begin{myenumerate}{Case}
\item Suppose there exist two $q$-cliques in $\Gamma$ with a common matrix.

For any matrix $A$ in $\Gamma$, $\Phi$ maps its $q$-cliques into its $q$-cliques, that is, $\Phi(l_{{\bf x}}^A)=l_{f({\bf x})}^{A}$  where $f$ is map on the set $\{{\bf x}\ :\ \langle{\bf x}\rangle\in V_{\inv{A^{-1}}}\}$. Since $l_{{\bf x}_1}^A=l_{{\bf x}_2}^A$ for linearly dependent nonzero ${\bf x}_1$ and ${\bf x}_2$, $f$ induces a map $\widehat{f}$ on $V_{\inv{A^{-1}}}$ given by $\widehat{f}(\langle{\bf x}\rangle)=\langle f({\bf x})\rangle$.

We first show that $\widehat{f}$ is bijective, so all $q$-cliques of $A$ are in $\Gamma$. Suppose the opposite, that is, $\Phi(l_{{\bf x}}^A)=\Phi(l_{{\bf y}}^A)$ for two distinct $q$-cliques. Then, we claim that there are only two possibilities:
\begin{enumerate}
\item If $l_{{\bf x}}^A$ and $l_{{\bf y}}^A$ are $A$-orthogonal, then any pair of distinct $q$-cliques in $\Gamma$ that share a common matrix, denoted by $B$, are $B$-non-orthogonal. Moreover, $\Phi(l_{{\bf z}}^A)\neq \Phi(l_{{\bf w}}^A)$ for any pair of $q$-cliques of $A$ that are $A$-non-orthogonal.
\item If $l_{{\bf x}}^A$ and $l_{{\bf y}}^A$ are $A$-non-orthogonal, then any pair of distinct $q$-cliques in $\Gamma$ that share a common matrix, denoted by $B$, are $B$-orthogonal. Moreover, $\Phi(l_{{\bf z}}^A)\neq \Phi(l_{{\bf w}}^A)$ for any pair of distinct $q$-cliques of $A$ that are $A$-orthogonal.
\end{enumerate}

To prove claim (i), let $l_{{\bf x}}^A$ and $l_{{\bf y}}^A$ be $A$-orthogonal and assume that $q$-cliques $l_{{\bf x}_1}^B$ and $l_{{\bf y}_1}^B$ are distinct, $B$-orthogonal, and lie in $\Gamma$. Then, by Lemma~\ref{lemma-pari2-simetrija}, there is a matrix $P$ such that $PBP^{\ast}=A$, $Pl_{{\bf x}_1}^B P^{\ast}=l_{{\bf x}}^A$, and $Pl_{{\bf y}_1}^B P^{\ast}=l_{{\bf y}}^A$, so $\Psi(X):=\Phi(PXP^{\ast})$ is a nonbijective endomorphism of $\Gamma$ as $\Psi(l_{{\bf x}_1}^B)=\Psi(l_{{\bf y}_1}^B)$. This is a contradiction, since $\Gamma$ is a core. Assume now that $\Phi(l_{{\bf z}}^A)=\Phi(l_{{\bf w}}^A)$ for a pair of $A$-non-orthogonal $q$-cliques. Let $l_{{\bf z}_1}^C$ and $l_{{\bf w}_1}^C$ be the two $q$-cliques from the assumption of Case~1. As  proved, they are $C$-non-orthogonal, so by Lemma~\ref{lemma-pari2-simetrija} there is an invertible $Q$ such that $QCQ^{\ast}=A$, $Ql_{{\bf z}_1}^C Q^{\ast}=l_{{\bf z}}^A$, and $Ql_{{\bf w}_1}^C Q^{\ast}=l_{{\bf w}}^A.$ Hence, $\Psi'(X):=\Phi(QXQ^{\ast})$ is a nonbijective endomorphism of $\Gamma$, a contradiction.

The proof of claim (ii) is dual to the proof of claim~(i).

To continue, assume firstly that $l_{{\bf x}}^A$ and $l_{{\bf y}}^A$ are $A$-orthogonal. Then $n\geq 4$ by Lemma~\ref{lemma-pari2-simetrija}. We claim that $\widehat{f}$ is an endomorphism of $\inv{H(n-1,q^2)}$, i.e., an endomorphism of the complement of the point graph of the hermitian polar space that is defined by $\inv{A^{-1}}$.
In fact, if $\langle{\bf z}\rangle$ and $\langle{\bf w}\rangle$ are two adjacent vertices in $\inv{H(n-1,q^2)}$, then ${\bf z}^{\ast} A^{-1}{\bf w}\neq 0$, so (i) implies that $\Phi(l_{{\bf z}}^A)\neq \Phi(l_{{\bf w}}^A)$. Since both $q$-cliques $l_{f({\bf z})}^A=\Phi(l_{{\bf z}}^A)$ and $l_{f({\bf w})}^A=\Phi(l_{{\bf w}}^A)$ are in $\Gamma$, (i) shows that $f({\bf z})^{\ast} A^{-1} f({\bf w})\neq 0$, so $\widehat{f}(\langle {\bf z} \rangle)$ and $\widehat{f}(\langle {\bf w} \rangle)$ are adjacent in $\inv{H(n-1,q^2)}$ and $\widehat{f}$ is an endomorphism. Since $\widehat{f}(\langle{\bf x}\rangle)=\widehat{f}(\langle{\bf y}\rangle)$, it is nonbijective, a contradiction by Lemma~\ref{cameron}.

Assume now that $l_{{\bf x}}^A$ and $l_{{\bf y}}^A$ are $A$-non-orthogonal. By (ii) and the assumption of Case~1, there exist two $q$-cliques of some $B$ that are pairwise $B$-orthogonal. It follows from Lemma~\ref{lemma-pari2-simetrija} that $n\geq 4$. Moreover, all $q$-cliques of $A$ in $\Gamma$ are pairwise $A$-orthogonal, so they form a subspace in variety $V_{\inv{A^{-1}}}$. Hence, by Theorem~\ref{bose_podprostor}, $\Gamma$ contains at most $\frac{q^{2\lfloor\frac{n}{2}\rfloor}-1}{q^2-1}$ $q$-cliques of $A$. By Lemma~\ref{lemma-q-(q-1)}, there are $|V_{\inv{A^{-1}}}|$ $q$-cliques of $A$ in total, so by Theorem~\ref{kardinalnost} there is a $q$-clique of $A$ in $\Gamma$, into which at least
$$\frac{|V_{\inv{A^{-1}}}|}{\frac{q^{2\lfloor\frac{n}{2}\rfloor}-1}{q^2-1}}=
\left\{
\begin{array}{ll}q^{n-1}+1\ &\textrm{if}\ n\ \textrm{is even}\\
q^{n}+1\ &\textrm{if}\ n\ \textrm{is odd}
\end{array}
\right\}\geq 4$$
$q$-cliques of $A$ are mapped by $\Phi$. Choose four of these $q$-cliques: $l_{{\bf z}_1}^A, l_{{\bf z}_2}^A, l_{{\bf z}_3}^A, l_{{\bf z}_4}^A$. As they are mapped into the same $q$-clique, they are $A$-non-orthogonal by (ii). Let $A=RR^{\ast}$, $\Psi''(X):=\Phi(RXR^{\ast})$, and ${\bf w}_i:=R^{-1}{\bf z}_i$. Then, $q$-cliques $l_{{\bf w}_i}$ are pairwise non-orthogonal and $\Psi''(l_{{\bf w}_i})=\Psi''(l_{{\bf w}_j})$ for all $i,j$. By Lemma~\ref{lemma-mesano} there exists a triple $(a_2,a_3,a_4)\neq (0,0,0)$ such that
\begin{equation}\label{eq35}
\left(\sum_{i=1}^{4} a_i{\bf w}_i\right)^{\ast}\left(\sum_{i=1}^{4} a_i{\bf w}_i\right)=0
\end{equation}
and
\begin{equation}\label{eq34}
{\bf w}_1^{\ast}\left(\sum_{i=1}^{4} a_i{\bf w}_i\right)=0
\end{equation}
for all $a_1\in\FFq2$.  Pick a nonzero scalar among $a_2,a_3,a_4$. We may assume that $a_2$ is such.
By Lemma~\ref{lemma-main},
\begin{equation}\label{i15}
\Psi''(l_{a_3{\bf w}_3+ a_4{\bf w}_4})\subseteq \Psi''(l_{{\bf w}_4})=\Psi''(l_{{\bf w}_1}).
\end{equation}
Since $a_2{\bf w}_1^{\ast}{\bf w}_2+{\bf w}_1^{\ast}( a_3{\bf w}_3+ a_4{\bf w}_4)=0$ by~(\ref{eq34}), we deduce that
${\bf w}_1^{\ast}( a_3{\bf w}_3+ a_4{\bf w}_4)\neq 0$. Consequently Lemma~\ref{lemma-main} and (\ref{i15}) implies that \begin{equation}\label{i16}
\Psi''(l_{a_1{\bf w}_1+a_3{\bf w}_3+ a_4{\bf w}_4})\subseteq \Psi''(l_{{\bf w}_1})=\Psi''(l_{{\bf w}_2})
\end{equation}
for all $a_1$. Pick $a_1$ such that ${\bf w}_2^{\ast}(a_1{\bf w}_1+ a_3{\bf w}_3+ a_4{\bf w}_4)\neq 0$. Then
$\Psi''(l_{\sum_{i=1}^{4} a_i{\bf w}_i})\subseteq \Psi''(l_{{\bf w}_2})=\Psi''(l_{{\bf w}_1})$ by Lemma~\ref{lemma-main} and (\ref{i16}). By (\ref{eq35}), $l_{\sum_{i=1}^{4} a_i{\bf w}_i}$ is a $q$-clique of $I$, so $\Psi''(l_{\sum_{i=1}^{4} a_i{\bf w}_i})=\Psi''(l_{{\bf w}_1})$, that is, $\Phi(l_{\sum_{i=1}^{4} a_i{\bf z}_i}^A)=\Phi(l_{{\bf z}_1}^A)$. This contradicts~(ii), since~(\ref{eq34}) implies that $l_{\sum_{i=1}^{4} a_i{\bf z}_i}^A$ and $l_{{\bf z}_1}^A$ are $A$-orthogonal.

We have now shown that $\widehat{f}$ is bijective, that is, all $q$-cliques of $A$ are in $\Gamma$. Recall that the matrix $A$ in $\Gamma$ was arbitrary. Since matrices in a $q$-clique have the same determinant,  Proposition~\ref{lemma-povez-det} implies that $\Gamma$ is a subgraph induced by a set $\bigcup_{\nu\in J} \det^{-1} \nu$, for some $J\subseteq \FF$. If $J\neq \FF$ that is
$\Gamma\neq\hgl$, then there is some invertible hermitian matrix $D\notin \Gamma$. Since $\Phi(D)\in \Gamma$, it follows that $\det\Phi(D)\neq \det D$. Since the determinant on any $(q-1)$-clique attains all values from $\FF\backslash\{0\}$, each $(q-1)$-clique of $D$ contains a matrix whose determinant equals $\det\Phi(D)$. So all these matrices are pairwise nonadjacent, contained in $\Gamma$, and fixed by $\Phi$, since it is a retraction. On the contrary, the number of pairwise nonadjacent neighbors of $\Phi(D)$ with determinant equal to $\det\Phi(D)$ equals the numbers of $q$-cliques of $\Phi(D)$, which, by Lemma~\ref{lemma-q-(q-1)}, is smaller than the number of $(q-1)$-cliques of $D$ . Since $\Phi$ preserves adjacency, we get a contradiction, so $J=\FF$ and $\Gamma=\hgl$.

\item Suppose the assumption of Case~1 is not true, and there exist a $q$-clique and a $(q-1)$-clique in $\Gamma$ with a common matrix.

    By Proposition~\ref{lemma-povez-det}, the graph induced by $\det^{-1}(\lambda)$ is connected for any $\lambda\in\FF\backslash\{0\}$. Since, by the assumption, there are no pairs of $q$-cliques with a common matrix in $\Gamma$, it follows that $\Phi(\det^{-1}(\lambda))$ is a $q$-clique, so all matrices in it have the same determinant. Since there exists a $(q-1)$-clique in $\Gamma$, determinant on matrices in $\Gamma$ attains all values from $\FF\backslash\{0\}$, so $q$-cliques $\Phi(\det^{-1}(\lambda)); \lambda \in\FF\backslash\{0\}$ must be disjoint, since matrices in a single $q$-clique have the same determinant. Consequently, $\Gamma$ consists of $q-1$ $q$-cliques, with some additional edges between the $q$-cliques. Since three pairwise adjacent matrices lies always in a maximal clique, the determinants of these three matrices are either all the same or all different. Consequently, for $\lambda_1\neq \lambda_2$, there is at most one edge from any fixed matrix in $\Phi(\det^{-1}(\lambda_1))$ to the $q$-clique $\Phi(\det^{-1}(\lambda_2))$. By the assumption, there is some matrix in $\Gamma$ that is a member of a $q$-clique and $(q-1)$-clique. Since $\hgl$, and consequently $\Gamma$, is vertex transitive, the same holds for all matrices in $\Gamma$, so we have completely determined edges between $q$-cliques $\Phi(\det^{-1}(\lambda))$. Consequently, $\Gamma$ is determined as well (cf.~Figure~2).
\begin{figure}
\centering
\begin{tabular}{@{}cc@{}}
\psfrag{X}{{\tiny $1$}} \psfrag{Y}{{\tiny $2$}} \psfrag{Z}{{\tiny $3$}} \psfrag{W}{{\tiny $4$}}
\includegraphics[width=0.4\textwidth]{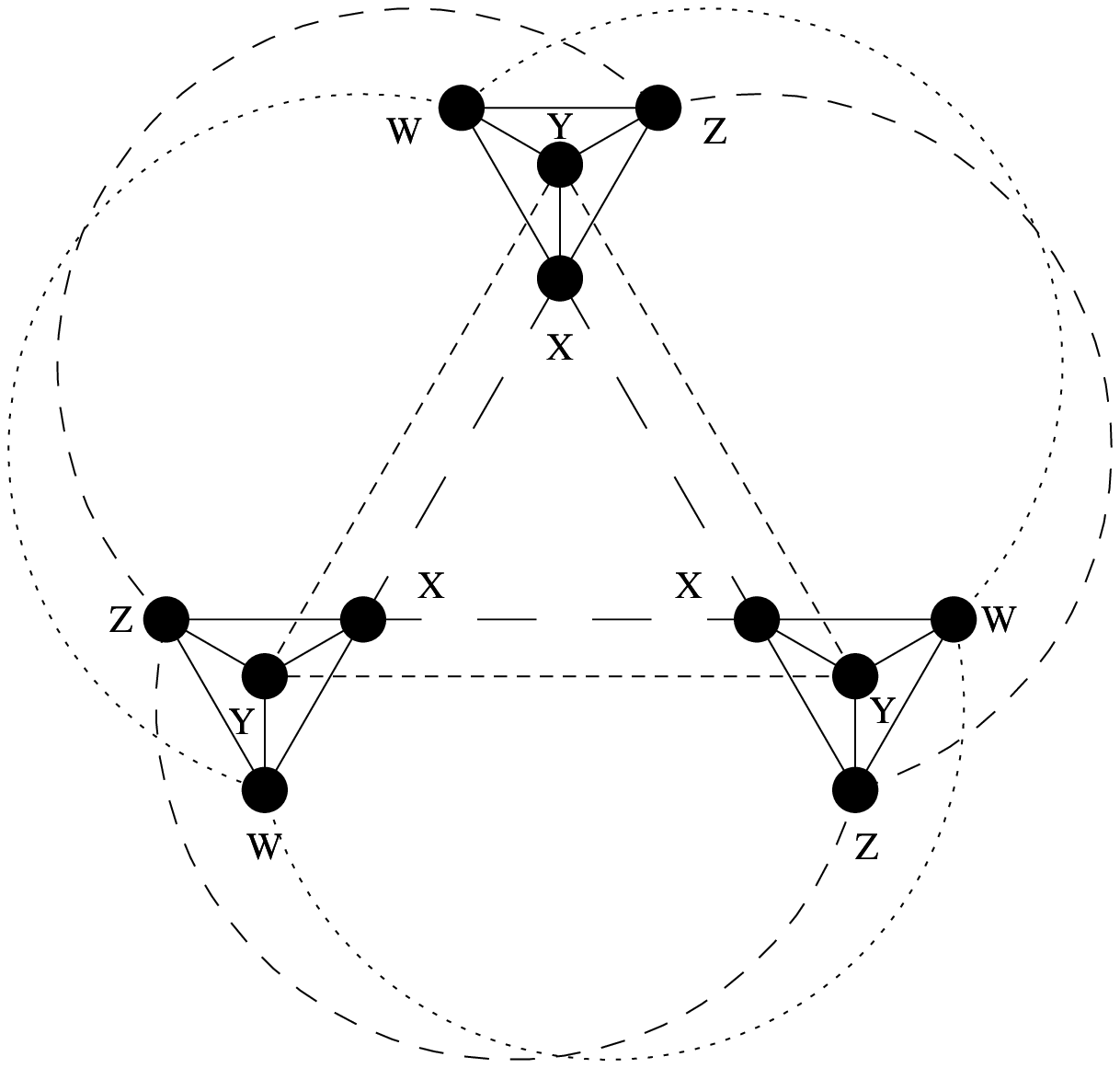}&
\psfrag{A}{{\tiny $1$}} \psfrag{B}{{\tiny $2$}} \psfrag{C}{{\tiny $3$}} \psfrag{D}{{\tiny $4$}}
\includegraphics[width=0.4\textwidth]{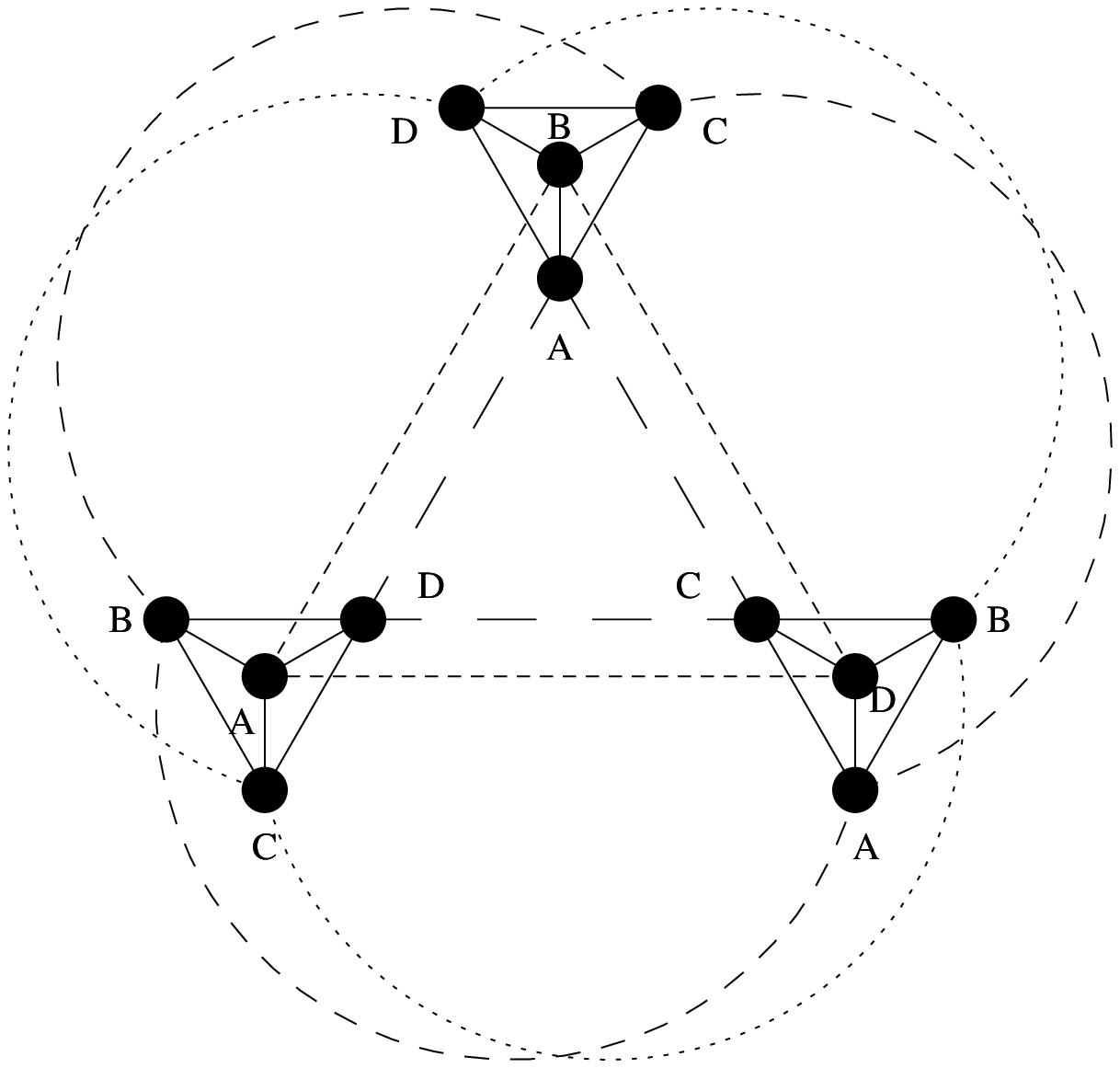}
\end{tabular}
\caption{The graph $\Gamma$ from Case~2 if $q=4$. An enumeration on the left and a 4-coloring on the right. Edges of 3-cliques are dashed/dotted.}
\end{figure}
We now show that $\chi(\Gamma)=q$, which is a desired contradiction by Lemma~\ref{klika-kromaticno}. Enumerate the vertices of a fixed $q$-clique by $1,2,\ldots,q$. Then enumerate the vertices in other $q$-cliques in the following way: matrices that form a $(q-1)$-clique, together with vertex 1 from the first $q$-clique, enumerate with 1;  matrices that form a $(q-1)$-clique, together with vertex 2 from the first $q$-clique, enumerate with 2; etc. To obtain a desired $q$-coloring, change the enumeration inside $q$-cliques as described below:
\begin{align*}
\textrm{1st}\ &q\textrm{-clique}:\ (1,2,\ldots,q)\mapsto (1,2,\ldots,q)\\
\textrm{2nd}\ &q\textrm{-clique}:\ (1,2,\ldots,q)\mapsto (q,1,2,\ldots,q-1)\\
\textrm{3rd}\ &q\textrm{-clique}:\ (1,2,\ldots,q)\mapsto (q-1,q,1,2,\ldots,q-2)\\
&\vdots\\
(q-1)\textrm{-th}\ &q\textrm{-clique}:\ (1,2,\ldots,q)\mapsto (3,4,\ldots,q,1,2)\\
\end{align*}

\item Suppose the assumptions of Case~1 and~2 are not true.

Then $\Gamma$ is a single $q$-clique, a contradiction by Lemma~\ref{klika-kromaticno}.\qedhere
\end{myenumerate}
\end{proof}

\end{document}